\documentclass[12pt]{amsart}

\usepackage{t1enc}
\usepackage[latin2]{inputenc}
\usepackage[english]{babel}
\usepackage{amssymb}
\usepackage{amsthm}
\usepackage{amsmath}
\usepackage{verbatim}
\usepackage{textcomp}
\usepackage{a4wide}

\usepackage[dvips]{graphicx}

\theoremstyle{plain}
\newtheorem{theorem}{Theorem}[section]
\newtheorem{lemma}[theorem]{Lemma}
\newtheorem{proposition}[theorem]{Proposition}
\newtheorem{corollary}[theorem]{Corollary}

\theoremstyle{definition}

\newtheorem{example}[theorem]{Example}
\newtheorem{examples}[theorem]{Examples}
\newtheorem{remark}[theorem]{Remark}

\DeclareMathOperator{\Aut}{Aut}

\DeclareMathOperator{\id}{id}

\DeclareMathOperator{\Age}{Age}
\DeclareMathOperator{\Flim}{Flim}
\DeclareMathOperator{\Cycl}{Cycl}
\DeclareMathOperator{\Betw}{Betw}
\DeclareMathOperator{\Sep}{Sep}

\newcommand{\N}{\mathbb N}
\newcommand{\im}{\mathrm{im}}
\newcommand{\emb}{\mathrm{e}}

\newcommand{\K}{\mathcal K}
\newcommand{\dom}{\textrm{dom}}

\date{{December 6, 2013}}

\begin{document}


\title{Topological dynamics of unordered Ramsey structures}

%%
%% Now edit the following to give your name and address:
%% 

\author{Moritz M\"{u}ller}
\address{Kurt G\"odel Research Center (KGRC), Vienna, Austria.}
\email{moritz.mueller@univie.ac.at}

\author{Andr\'as Pongr\'acz}
\address{Laboratoire d'Informatique de l'\'Ecole Polytechnique (LIX), Palaiseau, France. }
\email{andras.pongracz@lix.polytechnique.fr}
\thanks{The second author has received funding from the European Research Council under the European Community's Seventh Framework Programme (FP7/2007-2013 Grant Agreement no. 257039).}

\begin{abstract}
In this paper we investigate the connections between Ramsey properties of Fra\"{\i}ss\'{e} 
classes $\K$ and the universal minimal flow $M(G_\K)$ of the automorphism group $G_\K$
of their Fra\"{\i}ss\'{e} limits. As an extension of a result of Kechris, Pestov and Todorcevic \cite{kpt} 
we show that if the class $\K$ has finite Ramsey degree 
for embeddings, then this degree equals the size of $M(G_\K)$. 
We give a partial answer to a question of Angel, 
Kechris and Lyons \cite{akl} showing that 
if $\K$ is a relational Ramsey class and $G_\K$ is amenable, then $M(G_\K)$ admits a unique
invariant Borel probability measure that is concentrated on a unique generic orbit.
\end{abstract}

 \maketitle

\section{Introduction}

With a Fra\"iss\'e class of finite structures $\K$ one can associate in a natural way a topological group $G_\K$, namely, the automorphism group of the Fra\"iss\'e limit of~$\K$. 
For example, the
Fra\"iss\'e limit of finite dimensional vector spaces over a 
fixed finite field $F$ is the $\aleph_0$-dimensional vector space $V_{\infty,F}$ over 
$F$ with automorphism group $\textup{GL}(V_{\infty,F})$.
The groups of the form $G_\K$ are precisely the Polish groups that are non-archimedian in the sense 
that they have a basis at the identity consisting of open subgroups (\cite{kb}). 

In \cite{kpt} Kechris, Pestov and Todorcevic  developed a ``duality theory'' \cite[$\S4$(A)]{kechris} 
linking finite combinatorics of $\K$ with topological dynamics of $G_\K$, more precisely, it links 
combinatorial properties of $\K$ with properties of the universal minimal $G_\K$-flow $M(G_\K)$. For groups 
of the form $G_\K$ the flow $M(G_\K)$ is an inverse limit of metrizable $G_\K$-flows (cf.~\cite[T1.5]{kpt}), 
and in many interesting cases is metrizable itself. If so, $M(G_\K)$ 
either has the size of the continuum or else is finite \cite[$\S 1$(E)]{kpt}.
An extreme case is that $M(G_\K)$ is a single point, that is, $G_\K$ is extremely amenable.
It is shown in \cite{kpt} that for ordered $\K$ this happens if and only if $\K$ is Ramsey.
%Assuming $\K$ to be ordered is not too much of a restriction because $G_\K$, if extremely amenable, automatically preserves some linear order (see e.g.~\cite[P4.3]{kpt}).
For example, $V_{\infty,F}$ together with the so-called ``canonical order'' has an extremely amenable automorphism group.

We give a characterization of $M(G_\K)$ having an arbitrary finite cardinality in terms of Ramsey properties of $\K$. 
Namely, we use Fouch\'e's Ramsey degrees \cite{fouche,fouche2,fouche3} and show that  $M(G_\K)$ has finite 
size $d$ if and only if $\K$ has Ramsey degree $d$ (Theorem~\ref{thm:char}). We do not 
assume $\K$ to be ordered, but use Ramsey degrees {\em for embeddings} 
instead (see~e.g.\ \cite{nesetril,bodirsky}). These coincide with the usual Ramsey degrees on rigid 
structures, so our characterization generalizes the mentioned result of~\cite{kpt} and so does its proof. As a corollary we get (Corollary~\ref{cor:asymptotic})
that Ramsey degrees for embeddings are asymptotic in the sense that 
%the degree of $A\in\K$ equals the minimal degree of any $B\in\K$ into which $A$ embeds (Corollary~\ref{cor:asymptotic}). 
all structures in $\K$ have degree at most $d$ if all {\em large enough} structures have degree at most $d$ (i.e.\ every structure 
embeds into one of degree at most $d$).

Given an appropriate (unordered) class $\K$ one can first produce a so-called reasonable order expansion $\K^*$ whose Fra\"iss\'e 
limit expands the limit of $\K$ by a (linear) order $<^*$. The group $G_\K$ acts naturally on 
orders and one gets a $G_\K$-flow $X_{\K^*}$ as the orbit closure $\overline{G_\K\cdot <^*}$. Again, as shown in \cite{kpt}, 
minimality of this flow corresponds to a combinatorial property of $\K^*$ called the ordering property (cf.~\cite{nesetril}), 
and indeed $X_{\K^*}$ is $M(G_\K)$ if and only if $\K^*$ additionally is Ramsey.\footnote{See \cite{nguyen2} for a discussion of how to characterize universality alone.} 
Moreover, the Ramsey degree of $A\in\K$ equals the number of non-isomorphic order expansions it 
has in $\K^*$(\cite[$\S10$]{kpt},\cite[$\S 4$]{nguyen}).

For example, the universal minimal $\textup{GL}(V_{\infty,F})$-flow is the orbit closure of the canonical order. 
This canonical order is forgetful in the sense that any finite dimensional $F$-vector space gets up to isomorphism 
only one order expansion, so Ramsey degrees are 1 in this case. The Ramsey degrees {\em for embeddings} 
on the other hand are unbounded (cf.~Corollary~\ref{cor:emb2}). In general, the relationship between the two degrees is not trivial. 
We show that if a Ramsey class in a relational language has finite Ramsey degree for embeddings, then this degree must be 
a power of 2 (Theorem~\ref{thm:mystery}).

Recently, Angel, Kechris and Lyons \cite{akl} extended the duality theory to other important properties of $M(G_\K)$, 
namely whether or not there is a (unique) $G_\K$-invariant Borel probability measure on $M(G_\K)$. 
In this case, the group $G_\K$ is called amenable (uniquely ergodic), and this happens if and 
only if all minimal $G_\K$-flows admit such a (unique)
measure (\cite[P8.1]{akl}). 
%It is known that so-called Hrushovski classes $\K$ have amenable $G_\K$ (cf.~\cite[P6.4]{kr}). 
%If $\K$ is Hrushovski and $\K^*$ is as above such that $X_{\K^*}$ is the universal minimal $G_\K$-flow,  
%Kechris et al.~\cite{akl} show that
% unique ergodicity of $G_\K$ is equivalent to $\K^*$ satisfying a certain quantitative version of the ordering property. 
For example, $\textup{GL}(V_{\infty,F})$ is uniquely ergodic.

The $G_\K$-flows $X_{\K^*}$ have a generic (i.e.\ comeager) orbit $G_\K\cdot <^*$ which is in 
fact dense $G_\delta$ \cite[P14.3]{akl}. In many examples, 
a
%the 
$G_{\K}$-invariant measure on $M(G_{\K})$, if exists, 
turns out to be concentrated on this generic orbit. However, answering a question in~\cite[Q15.3]{akl}, 
Zucker \cite[T1.2]{zucker} showed that the measure on $M(\textup{GL}(V_{\infty,F}))$ is {\em not} concentrated on the generic orbit.

We show that such counterexamples rely on the language containing function symbols. More precisely,
we show that if $\K$ is Ramsey over a relational language and $G_\K$ is amenable, then $G_\K$ is uniquely ergodic 
and the unique $G_\K$-invariant 
Borel probability measure on $M(G_\K)$ is indeed concentrated on a dense $G_\delta$ orbit (Theorem~\ref{thm:concentr}).

%Both directions of the duality are of interest. In \cite{kpt} it is mainly used to produce many examples of extremely amenable groups and to calculate universal minimal flows of automorphism groups, relying on Ramsey results. But ``it seems interesting to reverse the roles''\cite[$\S 6$(A)]{kpt} hoping for new tools for proving Ramsey properties. This perspective is expressed in \cite{nguyen2}, and exemplified in \cite[C6.8]{kpt}. This paper intends too to use the duality for a better understanding of Ramsey combinatorics.

%Note that the characterization of extreme amenability mentioned above needs the Fra\"iss\'e class $\K$ to be ordered. For unordered classes $\K$ Nguyen van Th\'e \cite{nguyen} showed that a similar result holds for classes $\K$ of rigid structures. We generalize these results, showing that  $M(G_\K)$ has some arbitrary finite cardinality $d$ if and only if $\K$ has Ramsey degree {\em for embeddings} at most $d$ (Theorem~\ref{thm:char}). Ramsey properties for embeddings have been considered in \cite{nesetril} and used in \cite{bodirsky}. On classes of rigid structres they conincide with the Ramsey degrees as introduced by Fouch\'e~\cite{fouche} and further discussed in~\cite[$\S 10$]{kpt}. For a rigid Fra\"iss\'e class $\K$ we show that Ramsey degrees are asymptotic in the sense that whenever every structure $A\in\K$ embeds into some $B\in\K$ of Ramsey degree at most $d$ in $\K$, then $A$ too has Ramsey degree at most $d$ in $\K$ (Corollary~\ref{cor:asymptotic}). 

\section{Preliminaries}

\subsection{Notation} For $k\in \N$ we let $[k]$ denote $\{0,\ldots,k-1\}$ and understand $[0]=\emptyset$. 
If $X,Y$ are sets, $f$ a function from $X$ to $Y$, $n\in \N$ and $Z\subseteq X^n$ we write $f(Z)$ for the 
set $\{f(\bar x)\mid \bar x\in Z\}$ where $f(\bar x)$ denotes the tuple $ (f(x_0),\ldots, f(x_{n-1}))$ 
for $\bar x=(x_0,\ldots, x_{n-1})\in X^n$.
For $X_0\subseteq X$ we let $f\upharpoonright X_0$ denote the restriction of $f$ to $X_0$; for a relation $Z$ 
as above, $Z\upharpoonright X_0$ denotes $Z\cap(X_0^n)$. The identity on $X$ is denoted by $\id_X$.

%; if $f(Z)\subseteq Z$ we say $f$ {\em preserves} $Z$. 
 
\subsection{Fra\"iss\'e theory}\label{sec:fraisse}
Fix a countable language $L$. We let $A,B,\ldots$ range over ($L$-)struc\-tures. The distinction between 
structures and their universes are blurred notationally. We speak of {\em relational} structures and classes of structures if the underlying language $L$ is relational. We write $A\le B$ to indicate that there exists an embedding from $A$ into $B$, and we let $B^A$ denote the set of embeddings from $A$ into $B$. 

The {\em age} $\Age(F)$ of  a structure $F$ is the class of finitely generated structures which embed into $F$. A structure $F$ is {\em locally finite} if its finitely generated substructures are finite. 
For $A\in\Age(F)$ we call $F$  {\em $A$-homogeneous} if for all $a,a'\in F^A$ there is $g\in\textup{Aut}(F)$ such that $g\circ a=a'$. If $F$ is $A$-homogeneous for all $A\in\Age(F)$, it is {\em (ultra-)homogeneous}. 

A structure $F$ is {\em Fra\"iss\'e} if it is countably infinite, locally finite and homogeneous. The age $\K:=\Age(F)$ of a Fra\"iss\'e structure $F$ 
\begin{itemize}\itemsep=0pt
\item[--] is {\em hereditary:} for all $A,B$, if $A\le B$ and $B\in\K$, then $A\in\K$; 
\item[--] has {\em joint embedding:} for all $A,B\in\K$ there is $C\in\K$ such that both $A\le C$ and $B\le C$; 
\item[--] has {\em amalgamation:} for all $ A, B_0, B_1\in\K$ and $a_0\in B_0^A, a_1\in  B_1^A$ there are $ C\in\K$ and $b_0\in C^{B_0},b_1\in  C^{ B_1}$ such that $b_0\circ a_0=b_1\circ a_1$.
\end{itemize}
A class $\K$ of finite structures that has these three properties and for every $n\in\N$ contains a structure (with universe) of size at least $n$, is a {\em Fra\"iss\'e class}. The following 
is well-known~\cite[T4.4.4]{zieglerbuch}:

\begin{theorem}[Fra\"iss\'e 1954] \label{thm:fraisse}
For every Fra\"iss\'e class $\K$ there exists a Fra\"iss\'e structure $F$ with age $\K$. 
\end{theorem}

A standard back-and-forth argument shows that the structure $F$ in Theorem~\ref{thm:fraisse} is unique up isomorphism; 
it is called the {\em Fra\"iss\'e limit of $\K$} and denoted by $\Flim(\K)$. 

We mention some standard examples:

\begin{examples}\label{exas:fraisseclasses} The Fra\"iss\'e limit of the class of linear orderings is the rational order $(\mathbb Q,<)$. 
The Fra\"iss\'e limit of the class of finite Boolean algebras is the countable atomless Boolean algebra $B_\infty$. 
The Fra\"iss\'e limit of the class of finite graphs is the random graph $R$. The Fra\"iss\'e limit of 
the class of finite vector spaces over a fixed finite field $F$ is the vector space $V_{\infty,F}$ of 
dimension $\aleph_0$ over $F$.
\end{examples}

We refer to \cite{cameron,cherlin,mcph} as surveys on homogeneous structures.

\subsection{Ramsey degrees} Write ${B \choose A}$ for the set of substructures of $B$ which are isomorphic to $A$. 
Note that ${B'\choose A}\subseteq{C\choose A}$ whenever $B'\in{C\choose B}$.
If $k,d\in\N$ then $C\to (B)^A_{k,d}$ means that for every colouring $\chi:{C\choose A}\to [k]$ 
there exists $B'\in{C\choose B}$ such that $|\chi({B'\choose A})|\le d$. The {\em Ramsey degree} of $A$ in
 a class 
of structures $\K$ is the least $d\in\N$ such that for all $B\in \K$ and $k\ge 2$ there is $C\in\K$ 
such that $C\to (B)^A_{k,d}$ -- provided that such a $d$ exists; otherwise it is $\infty$. 
Taking the supremum over $A\in\K$ gives the Ramsey degree of $\K$, and 
the Ramsey degree of a structure $F$ is understood to be the Ramsey degree of $\Age(F)$; if 
this degree is 1, then $\K$ resp. $F$ are simply called {\em Ramsey}. 
%E.g.\ the class of finite Boolean algebras is Ramsey.

\begin{examples} $(\mathbb Q,<)$, $B_\infty$ and $V_{\infty,F}$ are Ramsey~\cite{kpt}. 
The random graph $R$ has Ramsey degree $\infty$; indeed, a finite graph $G$ has Ramsey 
degree $|G|!/|\Aut(G)|$ in the class of finite graphs~\cite[$\S10$]{kpt}.
\end{examples}

%A structure is {\em rigid} if it does not have nontrivial automorphisms, and a class of structures $\K$ is rigid if all its members are. Ne\v{s}et\v{r}il showed~\cite[T3.2]{nesetril}:
%
%\begin{lemma} \label{lem:amalg}
%Let $\K$ be a rigid, hereditary class of finite structures with joint embedding. If $\K$ is Ramsey, then it has amalgamation.
%\end{lemma}
%
%Another useful observation is (see~e.g.~\cite{nguyen1} for the case $d=1$ or Remark~\ref{rem:infinite2} below):
%
%\begin{lemma}\label{lem:infinite} 
%Let $F$ be a Fra\"iss\'e structure and $d\in\N$. Then $F$ has Ramsey degree at most $d$ if 
%and only if $F\to(B)^A_{k,d}$ for all $A,B\in\Age(F)$ and $k\in\N$.
%\end{lemma}

Ramsey degrees have been introduced by Fouch\'e in \cite{fouche}. We refer to 
the surveys \cite{grs, nesetrilsurvey} on Ramsey theory.

\subsection{Topological dynamics} With a Fra\"iss\'e class $\K$ we associate the topological group
$$
G_\K:=\Aut(\Flim(\K)),
$$
the identity having basic neighborhoods
$$
G_{(A)}:=\{g\in G_\K\mid g\upharpoonright A=\textup{id}_A\}
$$
for all finite substructures $A$ of $\Flim(\K)$. For any topological group $G$ a {\em $G$-flow} is a continuous action 
$a:G\times X\to X$ of $G$ on a compact Hausdorff space $X$. 
When the action is understood we shall refer to $X$ as a $G$-flow 
and write $g\cdot x$ or $gx$ for $a(g,x)$. For $Y\subseteq X$ we
write $G\cdot Y:=\bigcup_{g\in G}gY=\bigcup_{y\in Y}Gy$ where $Gy:=\{gy\mid g\in G\}$ denotes the {\em orbit of $y$} and 
$gY:=\{gy\mid y\in Y\}$.

\begin{example}\label{exa:lo} Let $G=\Aut(F)$ for a countable structure $F$. 
The {\em space of linear orders (on $F$)} is 
$\textit{LO}:=\{R\subseteq F^2\mid R \textrm{ is a linear order on } F\}$ with topology given by basic open sets
$\{R\mid \ R_0 \subseteq R\}$ for $R_0$ a linear order on a finite 
subset $A$ of $F$. This space is compact and Hausdorff, and a $G$-flow with 
respect to $(g,R)\mapsto g(R)$, the {\em logic action} of $G$ on $\textit{LO}$.
%which is defined by $a (g\cdot R) b \Leftrightarrow g^{-1}(a) R g^{-1}(b)$ for $g\in G, R\in\textit{LO}$.
\end{example}

A subset $Y\subseteq X$ is {\em $G$-invariant} if $G\cdot Y\subseteq Y$. 
Closed $G$-invariant subsets $Y$ are $G$-flows with respect to the restriction 
of the action. Such $G$-flows are {\em subflows} of $X$. 
The flow $X$ is {\em minimal} if $X$ and $\emptyset$ are its only subflows, 
that is, if and only if every orbit is dense. By Zorn's lemma, every $G$-flow contains a minimal subflow.
%
%If $X$ and $\emptyset$ are the only subflows, then $X$ is {\em minimal}. 
%The flow $X$ is minimal if and only if every orbit is dense, i.e.\ $\overline{G\cdot x}=X$ 
%for every $x\in X$. By Zorn's lemma, every $G$-flow contains a minimal subflow.
%
A {\em homomorphism (isomorphism)} of a $G$-flow $X$ into another $Y$ is a continuous (bijective) 
$G$-map $\pi:X\to Y$; being a $G$-map means that $\pi(g\cdot x)=g\cdot \pi(x)$ for all $g\in G,x\in X$.

The following is well-known (cf.~\cite[$\S 3$]{usp}). 

\begin{theorem} For every  Hausdorff  topological group $G$ there exists a minimal $G$-flow $M(G)$ which is {\em universal} in the sense that for every minimal $G$-flow $Y$ there is a homomorphism from $X$ into $Y$. Any two universal minimal $G$-flows are isomorphic. 
\end{theorem}

%We shall only deal with groups whose universal minimal flow is metrizable. The following is observed in \cite[$\S 1$(E)]{kpt}).
%
%\begin{proposition} Let $G$ be a Hausdorff topological group and assume that $M(G)$ is metrizable. Then either $M(G)$ has the size of the continuum or else is finite.
%\end{proposition}

An interesting case is that $|M(G)|=1$, equivalently, every $G$-flow $X$ has a fixed point, i.e. 
an $x\in X$ such that $G\cdot x=\{x\}$. In this case $G$ is called {\em extremely amenable}. 
Being {\em amenable} means that there exists a (Borel probability) measure $\mu$ on $M(G)$ which 
is {\em $G$-invariant} (i.e. $\mu(X)=\mu(g\cdot X)$ for every Borel $X\subset M(G)$ and $g\in G$). 
If there is exactly one such measure then
$G$ is {\em uniquely ergodic}. It is shown in \cite[P8.1]{akl} that for a uniquely ergodic  $G$ 
in fact every minimal $G$-flow has a unique $G$-invariant measure. 

We refer to \cite[$\S 1$]{kpt} for a survey on universal minimal flows.

\subsection{Duality theory} Let $<$ be a binary relation symbol.  
A class $\K^*$ of finite $L\cup\{<\}$-structures is {\em ordered} if each of its members 
has the form $(A,<^A)$ for a (linear) order $<^A$ (on $A$) and some finite $L$-structure $A$; the 
order $<^A$ is called a {\em $\K^*$-admissible} one (cf.~\cite{nesetril}). 

The following is \cite[T4.8]{kpt}.

\begin{theorem}\label{thm:extremeamenability} Assume that $\K^*$ is an ordered Fra\"iss\'e class.
Then $G_{\K^*}$ is extremely amenable if and  only if $\K^*$ is Ramsey.
\end{theorem}

Let $\K:=\{A\mid (A,<^A)\in \K^*\}$ be the {\em $L$-reduct} of $\K^*$; $\K^*$ is {\em reasonable} if 
for all $A,B\in\K$, all $a\in B^A$ and all $\K^*$-admissible orders $<^A$ on $A$ 
there is a $\K^*$-admissible order $<^B$ on $B$ such that $a(<^A)\subseteq <^B$, i.e. $a\in (B,<^B)^{(A,<^A)}$. 

\begin{lemma}\label{lem:reason}
Let $\K$ be a Fra\"iss\'e class and let $F=\Flim(\K)$. Then $\K^*=\Age(F, R)$ is reasonable for every order $R$ on $F$.
\end{lemma}
\begin{proof}
Let $A,B\in\K$, $a\in B^A$ and $<^A$ be a $\K^*$-admissible order on $A$. Let $a_0\in (F, R)^{(A, <^A)}$ and $b\in F^B$. In particular, $a_0\in F^A$ and $b\circ a \in F^A$, and then by homogeneity of $F$ there exists an $\alpha\in\Aut(F)$ such that $\alpha\circ b\circ a = a_0$. We define 
$$<^B := (b^{-1}\circ \alpha^{-1})(R\upharpoonright (\alpha\circ b)(B))$$

We need to show that $a^{-1}(<^B\upharpoonright a(A)) = <^A$. We have that 
\begin{multline*}
a^{-1}(<^B\upharpoonright a(A)) = a^{-1}((b^{-1}\circ \alpha^{-1})(R\upharpoonright (\alpha\circ b)(B))\upharpoonright a(A)) =\\ a^{-1}((b^{-1}\circ \alpha^{-1})(R\upharpoonright (\alpha\circ b)(a(A)))) = a_0^{-1}(R\upharpoonright a_0(A))= <^A
\end{multline*}

The last equality holds as $a_0\in (F, R)^{(A, <^A)}$.
\end{proof}

The following is \cite[P5.2,~T10.8]{kpt}. Recall that $\textit{LO}$ denotes the space of orders (Example~\ref{exa:lo}).

\begin{theorem}\label{thm:minflow} Let $\K^*$ be a reasonable ordered Fra\"iss\'e class in the language $L\cup\{<\}$ 
and $\K$ its $L$-reduct. 
\begin{enumerate}
\item Then $\K$ is  Fra\"iss\'e and $\Flim(\K^*)=(\Flim(\K),<^*)$ for some linear order $<^*$. 
\item Let $X_{\K^*}:=\overline{G_\K\cdot <^*}$ be the orbit closure of $<^*$ in the logic action of $G_\K$ on  
$\textit{LO}$. Then $X_{\K^*}$ is the universal minimal $G_\K$-flow if and only if $\K^*$ is Ramsey and has the ordering property.
\end{enumerate}
\end{theorem}

\begin{comment}
%
%\begin{examples} Declare an ``admissible'' order of a finite Boolean algebra to be one that is in a certain sense %`canonically' induced by an arbitrary linear order on its atoms. Declare an ``admissible'' order of a finite vector space over %a fixed finite field to be one that is in a certain sense `canonically' induced by an arbitrary linear order on a basis. Finally,
%declare an ``admissible'' order of a finite graph to be an arbitary linear order. See \cite[$\S6$]{kpt} for details. 
%In all cases one gets reasonable ordered Fra\"iss\'e classes which moreover are Ramsey and have the ordering property.
%\end{examples}
%
 A Fra\"iss\'e class
$\K$ is {\em Hrushovski} if for all $A\in\K$ there exists
$B\in\K$ such that every partial isomorphism of $A$ extends to an
automorphism of $B$; here, a {\em partial isomorphism} of $A$ is an
isomorphism between substructures of $A$. 
Kechris and Rosendal \cite[P6.4]{kr} showed that in this case $G_\K$ is compactly 
approximable and this is known to imply amenability (see e.g.~\cite[449C(b)]{fremlin}). Hence,

%In this case $\Flim(\K)$ is known to be compactly approximable (see~\cite[P6.4]{kr}) and this is known to imply amenability:

\begin{proposition}\label{prop:hrush}
If $\K$ is a Hrushovski Fra\"iss\'e class, then $G_\K$ is amenable.
\end{proposition}
\end{comment}

That $\K^*$  has the {\em ordering property} means that for all $A\in\K$ there is a $B\in\K$ such that
$(A,<^A)\le(B,<^B)$ for all $\K^*$-admissible orders $<^A$ on $A$ and $<^B$ on $B$.

In \cite {akl} Kechris et al. showed that a certain quantitative version of the ordering property characterizes unique ergodicity 
for so-called Hrushovski classes. Here, we shall only need the following \cite[P9.2]{akl}.
%
%Kechris et al.~\cite{akl} characterized unique ergodicity for Hrushovski classes.
%We shall only need the following proposition \cite[P9.2]{akl}. 

\begin{proposition} \label{prop:randomorder} Let $\K^*$ be a reasonable ordered Fra\"iss\'e 
class which is Ramsey and satisfies the ordering property, and let $\K$ be its $L$-reduct. 
Then $G_\K$ is amenable (uniquely ergodic)
if and only if there exists a
%(exactly one) 
consistent random $\K^*$-admissible ordering $(R_A)_{A\in\K}$ (and for every other consistent 
random $\K^*$-admissible ordering $(R'_A)_{A\in\K}$ we 
have that $R_A$ and $R'_A$ have the same distribution for every $A\in\K$). 

Indeed, if  $(R_A)_{A\in\K}$ is a random $\K^*$-admissible ordering, then
there is a $G_\K$-invariant Borel probability measure $\mu$ on $X_{\K^*}$ 
such that for every $A\in\K$ and every 
$\K^*$-admissible ordering $<$ on $A$ we have\footnote{Given a random variable we always use $\Pr$ to 
denote the probability measure of its underlying probability space.} 
$\mu(U(<))=\Pr[R_A=<]$ where $$U(<):=\{R\in X_{\K^*}\mid R\upharpoonright A=<\}.
$$
\end{proposition}

A {\em random $\K^*$-admissible ordering} is a family $(R_A)_{A\in\K}$ of random variables 
such that each $R_A$ takes values in the set of $\K^*$-admissible 
orders on~$A$. It is {\em consistent} if 
for all $A,B\in\K$ and $a\in B^A$ the random variables $a^{-1}(R_B\upharpoonright \im(a))$ and $R_A$ have the same distribution.

%A {\em consistent random $\K^*$-admissible ordering} is a family $(R_A)_{A}$ of random variables 
%such that for every $A\in \K$, $R_A$ takes values in the finite set of $\K^*$-admissible 
%orders on~$A$, and for all $a\in B^A,B\in\K$ we have 
%$a^{-1}(R_B\upharpoonright A)\sim R_A$ (i.e. the two variables 
%take any given value with the same probability).

\begin{examples}\label{exas} In \cite[$\S6$]{kpt} the reader can find constructions of reasonable 
ordered Fra\"iss\'e classes $\K^*$ whose reduct $\K$ is any of the classes mentioned in Example~\ref{exas:fraisseclasses}; 
in all these cases $\K^*$ is Ramsey and has the ordering property. By Theorem~\ref{thm:extremeamenability} one sees
that the automorphism groups of $(\mathbb Q, <)$ and of certain ordered versions of $B_\infty,R,V_{\infty,F}$ are 
extremely amenable~\cite{kpt}. Theorem~\ref{thm:minflow} allows  to calculate the universal minimal flows of the automorphism 
groups of $B_\infty,R$ and $V_{\infty,F}$. $\Aut(B_\infty)$ is not amenable, while $\Aut(R)$ and $\Aut(V_{\infty,F})$ are uniquely ergodic~\cite{akl}.
\end{examples}

\section{Automorphism groups with finite universal minimal flows}

Theorem~\ref{thm:extremeamenability} characterizes the condition that the universal minimal flow has size 1. In this section we provide a similar characterization for the condition that it has an arbitrary finite size.
To this end we consider Ramsey degrees {\em for embeddings}. The main result in this section reads:

\begin{theorem}\label{thm:char}
Let $d\in\N$ and $\K$ be a Fra\"iss\'e class.
The following are equivalent.
\begin{enumerate}
\item $M(G_\K)$ has size at most $d$;
\item  the Ramsey degree for embeddings of $\K$ is at most $d$.
\end{enumerate}
\end{theorem}

We start with some preliminary observations concerning finite universal minimal flows in Section~\ref{sec:finiteflow}. 
In Section~\ref{sec:remb} we define Ramsey degrees for embeddings and discuss their relationship to Ramsey 
degrees. The results proved in Sections~\ref{sec:finiteflow} and \ref{sec:remb} are mainly folklore.
In Section~\ref{sec:proof} we prove the result above and in 
%the final 
Section~\ref{sec:cor}  we note some corollaries.

%%%%%%%%%%%%%%%%%%%%%%%%%%%%%%%%%%%%%%%%%%%%%%%%%%
%%%%%%%%%%%%%%%%%%%%%%%%%%%%%%%%%%%%%%%%%%%%%%%%%%%

\subsection{Finite universal minimal flows}\label{sec:finiteflow}

\begin{lemma}\label{lem:orb}
 Let $G$ be a topological Hausdorff group and $d\in\mathbb N$. 
Then $M(G)$ has size at most $d$ if and only if every nonempty $G$-flow has an orbit of size at most $d$.
\end{lemma}

\begin{proof}
Assume that $|M(G)|\le d$, and let $X$ be a nonempty $G$-flow. Then there is a minimal subflow $X'$ of $X$ and 
a homomorphism $\pi$ of $M(G)$ onto $X'$. Thus $|X'|\leq d$.

Conversely, if every nonempty $G$-flow has an orbit of size at most $d$, then so does $M(G)$. Since $M(G)$ is 
minimal, this orbit is dense in $M(G)$, so it is equal to $M(G)$ by finiteness.  
\end{proof}

\begin{lemma}\label{lem:todorpart}
 Let $G$ be a topological Hausdorff group and $H$ an extremely amenable closed subgroup of $G$ with finite index. Then $H$ is a 
normal clopen subgroup of $G$ and $M(G)$ is isomorphic to the action of $G$ on $G/H$ by left multiplication.
\end{lemma}

\begin{proof} Clearly, a closed subgroup of finite index is open. We first show that $G/H$ is the universal minimal $G$-flow.
Since $H$ is open $G/H$ is discrete, and as $|G:H|$ is finite, $G/H$ is compact. Hence, $G/H$ is a $G$-flow. 
It is minimal, because $G$ acts 
transitively on $G/H$. If $Y$ is an arbitrary $G$-flow, 
then its restriction to $H$ is an $H$-flow, so it has a fixed point $y\in Y$. Then $gH\mapsto gy$ is
a homomorphism from $G/H$ into $Y$. 

As $gHg^{-1}$ is a closed subgroup of finite index for every $g\in G$, 
so is $H'=H\cap gHg^{-1}$. As above, we see that $G/H'$ is a minimal $G$-flow. By universality of $G/H$ 
there exists a surjection from $G/H$ onto $G/H'$, so $|G:H'|\le |G:H|$. Thus $H= gHg^{-1}$ for every $g\in G$, 
that is, $H$ is normal.
\end{proof}

\begin{proposition}\label{prop:todor}  Let $G$ be a topological Hausdorff group and $d\in\mathbb N$. Then 
$M(G)$ has size~$d$ if and only if 
$G$ has an extremely amenable, open, normal subgroup of index $d$.
\end{proposition}

%\begin{proposition}\label{prop:todor} Let $d\in\N$ and $\K$ be a Fra\"iss\'e class. The following are equivalent.
%\begin{enumerate}
%\item $M(G_\K)$ has size $d$; 
%\item $G_\K$ has an extremely amenable, open, normal subgroup of index $d$.
%\end{enumerate}
%\end{proposition}

\begin{proof}
%(2) $\Rightarrow$ (1) follows from Lemma~\ref{lem:todorpart}.
%
%(1) $\Rightarrow$ (2). Write $G:=G_\K$ and $X:=M(G_\K)$.
The backward direction follows from  Lemma~\ref{lem:todorpart}. Conversely, assume that  $X:=M(G)$ has size $d$.
For $x\in X$ let $H_x\leq G$ be the stabilizer of $x$.
Then there is a bijection between the set of left cosets of $H_x$ and the orbit $G\cdot x$. Since $G\cdot x$ is finite, 
$G\cdot x=\overline{G\cdot x}$, so $G\cdot x=X$ by minimality.
Hence, $|G:H_x|=|X|=d$.  As $H_x$ is closed and of finite index, so is 
$N:=\bigcap_{x\in X}H_x$, and hence $N$ is clopen.
%, it is open.
%Then $N=\bigcap_{x\in X}H_x$ is clopen. 
Since $N$ is the pointwise stabilizer of~$X$, it is normal. 
%We claim that $|G:N|=d$. As $N$ is the intersection of $d$  
 %subgroups of index~$d$, $|G:N|$ is finite and it is at least $d$. Since $N$ is open of finite index, $G/N$ is discrete and finite, so $G/N$ is compact. 
%But $G$ acts transitively on $G/N$ by left multiplication, so $G/N$ is a minimal $G$-flow. Since $X$ is universal, 
%there is a surjection of $X$ onto $G/N$, and $|G:N|\le d$ follows. 

%We are left to show that $N$ is extremely amenable. 
Let $Y$ be a minimal $N$-flow. %We prove that $|Y|=1$.
Let $\tau: G/N\rightarrow G$ be a function with $\tau(hN)\in hN$.
Define $a: G \times G/N \rightarrow N$ by setting
$$
a(g, hN):= \tau(hN)^{-1}\cdot g^{-1}\cdot \tau(ghN).
$$
A straightforward calculation shows that $a$ satisfies the so-called cocycle identity, that is, for all $g_1,g_2,h\in G$ 
\begin{equation}\label{eq:cocycle}
a(g_1 g_2, hN) = a(g_2, hN)\cdot a(g_1, g_2 hN). 
\end{equation}

We can construct an action of $G$ on $(G/N \times Y)$ by
$$
(g,(hN, y)) \mapsto (ghN, a(g, hN)^{-1}\cdot y).
$$
That this indeed defines a group action follows directly from~\eqref{eq:cocycle}.
The action is continuous and $(G/N \times Y)$ is compact, so $(G/N \times Y)$ is a $G$-flow.

%We prove that  $(G/N \times Y)$ is minimal. %Then by universality of $X$ we get
%a surjection from $X$ onto  $(G/N \times Y)$; thus $|G/N\times Y|\le d$, so $|Y|=1$ (as $|G/N|=d$).  
%To this end it is shown that for every $h\in G$ and $y\in Y$ the orbit $G\cdot (hN,y)$ is dense in $(G/N\times Y)$. 
Let $h\in G,y\in Y$ be arbitrary.
We show that
\begin{equation}\label{eq:section}
Y(h,y):=\{a(n, hN)^{-1}\cdot y\mid n\in N\}\textup{ is dense in } Y.
\end{equation}
Indeed, as $N$ is normal, we have $Y(h,y)=\tau(hN)^{-1} \cdot N \cdot \tau(hN) \cdot y = N\cdot y$. Since $Y$ is a minimal $N$-flow, the orbit $N\cdot y$ is dense in $Y$. 

The orbit $G\cdot (hN,y)$ contains $N\cdot (ghN,y')$ for every $g\in G$ and $y':=a(g, hN)^{-1}\cdot y$. But
$N\cdot (ghN,y')=\{(nghN, a(n, hN)^{-1}\cdot y')\mid n\in N\}=\{ghN\}\times Y(h,y')$, where the last equality 
holds because $N$ is normal. So the orbit $G\cdot (hN,y)$ contains $\bigcup_{g\in G}(\{gN\}\times Y(h,y_g))$ for certain $y_g$'s, 
and this set is dense in $(G/N\times Y)$ by~\eqref{eq:section}. Thus $(G/N\times Y)$ is a minimal $G$-flow.

By the universality of $X$ there exists a surjection from $X$ onto  $(G/N \times Y)$. 
In particular, $|G/N\times Y|\le d$. By definition  of $N$ we have $|G:N|\ge d$, 
so $|Y|=1$, $|G/N|=d$.
%%% NEW TEXT:
This means $N$ is extremely amenable and has index $d$ in $G$.
%%% OLD TEXT:
%, and consequently, $N$ 
%is extremely amenable and $N=H_x$ for all $x\in X$. We are done by Lemma~\ref{lem:todorpart}.
\end{proof}

\begin{example}\label{ex:flowd}
For $d\in \mathbb{N}$ let $G^*$ be the automorphism group of $(\mathbb Q, <, 0,1, \ldots, d-1)$, the structure  
with universe $\mathbb Q$ that interprets for all $i\in[d]$ a constant by $i$ and a 
binary relation symbol $<$ by the rational order. Let $G$ be the group 
generated by $G^*$ and the permutation $\alpha=(0\ 1 \ldots\ d-1)$. 
This is a closed subgroup of the group of all permutations of $\mathbb Q$, so $G=G_\K$ for some Fra\"iss\'e class $\K$ 
(see e.g.~\cite{kb}). 
Since $\alpha$ commutes with $G$, $G^*$ is normal in~$G$. Moreover, $G^*$ has index $d$ in $G$, and it follows from
\cite[L13]{bpt} (see also \cite[P24]{bodpin}) that~$G^*$ is extremely amenable. 
By Lemma~\ref{lem:todorpart}, $|M(G)|=|G/G^*|=d$.
%Thus Proposition~\ref{prop:todor} implies that $M(G)$ has size $d$.
\end{example}

\begin{example} 
Let $G$ be the automorphism group of $(\mathbb Q,E_d,<)$ where $<$ is the rational order and $E_d$ is an equivalence relation with $d$ classes each of which 
is dense in $(\mathbb Q,<)$. Let $H$ be the subgroup of $G$ consisting of those automorphisms that preserve each of the classes.
It is shown in \cite[T8.4]{kpt} that $H$ is extremely amenable and of index $d!$ in $G$.
By Lemma~\ref{lem:todorpart}, $|M(G)|=|G/H|=d!$.
%Let $\K$ be the Fra\"iss\'e class of all finite equivalence relations having at most~$d$ classes. It is shown in \cite[T8.4]{kpt} that $M(G_\K)$ has size $d!$.
\end{example}

\subsection{Ramsey degrees for embeddings}\label{sec:remb}
Let $k,d\in\N$ and $\K$ be a class of finite structures. Then $C\hookrightarrow (B)^A_{k,d}$ means that for every colouring $\chi:C^A\to [k]$ there exists a $b\in C^B$ such that $|\chi(b\circ B^A)|\le d$.  Naturally here, $b\circ B^A$ denotes $\{b\circ a\mid a\in B^A\}$.
The {\em Ramsey degree for embeddings} of $A$ in $\K$ is the least $d\in\N$ such 
that for all $B\in \K$ and $k\ge 2$ there is a $C\in\K$ such that $C\hookrightarrow (B)^A_{k,d}$ --
provided that such a $d$ exists; otherwise it is $\infty$. 
Taking the supremum over $A\in\K$ gives the Ramsey degree for embeddings of $\K$. If this degree is~1 we call $\K$ {\em Ramsey for embeddings}.
% and the Ramsey degree for embeddings of a structure $F$ is understood to be the Ramsey degree for embeddings of $\Age(F)$.\medskip

%To start we note analogues of Lemmas~\ref{lem:infinite} and~\ref{lem:amalg}.

\begin{lemma}\label{lem:infinite2} Let $d\in\N$, $\K$ be a Fra\"iss\'e class, $F=\Flim(\K)$ and $A\in\K$. 
The Ramsey degree for embeddings of $A$ in $\K$ is at most $d$ if and only if $F\hookrightarrow (B)^A_{k,d}$ 
for all $B\in\K$ and~$k\ge 2$.
\end{lemma}

\begin{proof} Assume that the Ramsey degree for embeddings of $A$ in $\K$ is at most $d$. Let $B\in\K, k\ge 2$ and 
$\chi:F^A\to[k]$. We are looking for $b'\in F^B$ such that $|\chi(b'\circ B^A)|\le d$.
Choose $C\in\K$ such that $C\hookrightarrow (B)_{k,d}^A$. Choose a $c\in F^C$ and let 
$\chi':C^A\to[k]$ map $a\in C^A$ to $\chi(c\circ a)$. By $C\hookrightarrow (B)_{k,d}^A$ there is a $b\in C^B$ 
such that $|\chi'(b\circ B^A)|\le d$, i.e. $|\chi(c\circ b\circ B^A)|\le d$. Then $b':=c\circ b\in F^B$ is as desired.

Assume  that there is an $A\in\K$ whose Ramsey degree for embeddings is bigger than $d$. Choose  $B\in\K, k\ge 2$ such 
that for every finite substructure $C$ of $F$ there is a colouring 
%%% NEW
$\chi:C^A\to[k]$ which is {\em good for $C$}, i.e.\ $|\chi(b\circ B^A)|> d$ for all $b\in C^B$. 
%%% OLD
%{\em good for $C$}, i.e. a colouring $\chi:C^A\to[k]$ such that $|\chi(b\circ B^A)|> d$ for all $b\in C^B$. 
The set $G(C):=\{\chi\in [k]^{F^A}\mid \chi\upharpoonright C^A\textrm{ is good for } C\}$ is nonempty and closed in $[k]^{F^A}$ carrying 
the product topology with $[k]$ being discrete. Given finitely many such sets $G(C_1),\ldots,G(C_n)$ their intersection 
contains the nonempty set $G(C)$ where $C$ is the substructure generated by $C_1\cup\ldots\cup C_n$ in $F$ 
(note that $C$ is finite by local finiteness of $F$). Since $[k]^{F^A}$ is compact, 
$\bigcap_CG(C)\neq\emptyset$ where $C$ ranges over the finite substructures of $F$. 
Any $\chi\in \bigcap_CG(C)$ is good for $F$, so $F\not\hookrightarrow (B)^A_{k,d}$. 
\end{proof}

We shall need the following result of Ne\v{s}et\v{r}il~\cite[T3.2]{nesetril}. 
We include the short proof.
%
%The following is well-known and due to  Ne\v{s}et\v{r}il~\cite[T3.2]{nesetril}. 
%We include the short proof for the convenience of the reader.

\begin{lemma}\label{lem:Ramam} Let $\K$ be a hereditary class of finite structures with joint embedding. If $\K$ is Ramsey for embeddings, then it has amalgamation.
\end{lemma}

\begin{proof} Let $A,B_0,B_1\in\K$ and $a_0\in B_0^A, a_1\in B_1^A$. Let $B\in\K$ and $b_0\in B^{B_0},b_1\in B^{B_1}$. Choose $C\in \K$ with $C\hookrightarrow(B)_{4,1}^{A}$. We claim that there exist $e_0\in C^{B_0},e_1\in C^{B_1}$ such that $e_0\circ a_0=e_1\circ a_1$. Consider the following colouring $\chi:C^A\to P(\{0,1\})$: for $a\in C^A$ the colour $\chi(a)\subseteq \{0,1\}$ contains $i\in\{0,1\}$ if and only if there exists an $e\in C^{B_i}$ such that $e\circ a_i=a$. Choose $b\in C^B$ such that $\chi(b\circ B^A)$ contains precisely one colour. 
Then this colour is $\{0,1\}$, because for $i\in\{0,1\}$ we have $i\in\chi(b\circ b_i \circ a_i)$ and $b\circ b_i \circ a_i\in b\circ B^A$. Let $a\in B^A$. Then $\chi(b\circ a)=\{0,1\}$, thus there are $e_0\in C^{B_0},e_1\in C^{B_1}$ such that $e_0\circ a_0=a=e_1\circ a_1$.
\end{proof}

%The previous proof mimicks Ne\v{s}et\v{r}il's argument, and indeed Lemma~\ref{lem:amalg} can be implied by Lemma~\ref{lem:Ramam} via:

\begin{remark}\label{rem:rigiddegrees} Clearly, $C\hookrightarrow (B)^A_{k,d}$ is equivalent to $C\rightarrow (B)^A_{k,d}$ when $A$ is rigid (i.e. $\Aut(A)=\{\id_A\}$). In particular, the Ramsey degree and the Ramsey degree for embeddings coincide for rigid structures. The following proposition generalizes this observation.
\end{remark}

\begin{proposition}\label{prop:emb} Let $d\in\N$, and let $\K$ be a class of finite structures. Let $A\in\K$ and $\ell=|\Aut(A)|$. 
The Ramsey degree for embeddings of $A$ in $\K$ is at most $d\cdot\ell$ if and only if 
the Ramsey degree of $A$ in $\K$ is at most $d$. 
\end{proposition}

\begin{proof} First assume that the Ramsey degree for embeddings of $A$ in $\K$ is at most $d\cdot\ell$. Let $B\in\K$ and $k\ge 2$. We are looking for a $C\in\K$ such that $C\to(B)_{k,d}^A$. By assumption we find some $C\in\K$ with $C\hookrightarrow(B)^A_{k,d\cdot\ell}$ and we claim that this $C$ is as desired. Let a colouring $\chi:{C\choose A}\to[k]$ be given.
For every $A'\in {C\choose A}$ there are precisely $\ell$ embeddings $a^{A'}_0,\ldots, a^{A'}_{\ell-1}\in C^A$ with image $A'$. Define $\chi':C^A\to[k]\times[\ell]$ to map $a\in C^A$ to $(i,j)$ for $i:=\chi(\im(a))$ and $j$ such that
$a=a_j^{\im(a)}$. Since  $C\hookrightarrow(B)^A_{k,d\cdot\ell}$ there is $b\in C^B$ such that $|\chi'(b\circ B^A)|\le d\cdot\ell$. 
Observe that $(i,j)\in\chi'(b\circ B^A)$ implies $\{i\}\times[\ell]\subseteq\chi'(b\circ B^A)$. Hence, there are (not necessarily distinct) $i_0,\ldots,i_{d-1}\in[k]$ such that $\chi'(b\circ B^A)=\{i_0,\ldots,i_{d-1}\}\times[\ell]$.
Clearly, $\im(b)\in{C\choose B}$ and we claim that $\chi({\im(b)\choose A})\subseteq \{i_0,\ldots,i_{d-1}\}$. Indeed, for $A'\in{\im(b)\choose A}$ there is an $a\in B^A$ such that $\im(b\circ a)=A'$, namely $a:= b^{-1}\circ a'$ for some isomorphism $a':A\rightarrow A'$. As ${\im(b)\choose A}\subseteq {C\choose A}$ we find $j\in[\ell]$ such that $a_j^{A'}=b\circ a$. Then $\chi'(b\circ a)=(\chi(A'),j)$, and in particular $\chi(A')\in \{i_0,\ldots,i_{d-1}\}$.

Conversely, assume that the Ramsey degree of $A$ in $\K$ is at most $d$. Let $B\in\K$ and $k\ge 2$ be given. 
By assumption there exists a $C\in\K$ such that $C\to(B)^A_{k^\ell,d}$. We claim that $C\hookrightarrow(B)_{k,d\cdot\ell}^A$.
Let $\chi:C^A\to[k]$ be a colouring and define $\chi':{C\choose A}\to[k]^{\ell}$ by setting
$\chi'(A'):=(\chi(a_0^{A'}),\ldots,\chi(a_{\ell-1}^{A'}))$ for $A'\in{C\choose A}$; here, for $A'\in{C\choose A}$ we let $a_0^{A'},\ldots,a_{\ell-1}^{A'}$ enumerate the embeddings in $C^A$ with image $A'$. Since $C\rightarrow(B)_{k^{\ell},d}^A$ there exists $B'\in{C\choose B}$ and $(i^{0}_0,\ldots, i^0_{\ell-1}),\ldots,(i^{d-1}_0,\ldots, i^{d-1}_{\ell-1})\in[k]^{\ell}$ 
such that $\chi'({B'\choose A})\subseteq\{(i^\nu_0,\ldots,i^\nu_{\ell-1})\mid \nu\in[d]\}$. 
Choose $b \in C^B $ with image $B'$. We claim that $\chi(b\circ B^A)\subseteq\{i^{\nu}_j\mid \nu\in[d],j\in[\ell]\}$.
Let $a\in B^A$. Then $b\circ a\in C^A$ and $\im(b\circ a)\in {B'\choose A}\subseteq {C\choose A}$. Choose $j\in[\ell]$ such that
$b\circ a=a_j^{\im(b\circ a)}$. Let $\nu\in[d]$ be such that 
$\chi'(\im(b\circ a))=(\chi(a_0^{\im(b\circ a)}),\ldots,\chi(a_{\ell-1}^{\im(b\circ a)}))=
(i^\nu_0,\ldots,i^\nu_{\ell-1})$. Hence, $\chi(b\circ a)=\chi(a_{j}^{\im(b\circ a)})=i^{\nu}_j$.
\end{proof}

\begin{corollary}\label{cor:emb2} Let $\K$ be a class of finite structures and $A\in \K$. Then
the Ramsey degree of $A$ in $\K$ is 1 if and only if the Ramsey degree for embeddings of $A$ in $\K$ is $|\Aut(A)|$. 
\end{corollary}

\begin{proof}
By Proposition~\ref{prop:emb} is suffices to show that the Ramsey degree for embeddings of $A$ in $\K$ is at least 
$\ell:=|\Aut(A)|$. Let $C\in \K$ be arbitrary. Using the notation from the previous proof, let $\chi:C^A\to[\ell]$ map 
$a\in C^A$ to the $j<\ell$ such that $a=a_j^{\im(a)}$. Then for $B:= A$ and 
every $b\in C^B$ we have $\chi(b\circ C^B)=[\ell]$.
\end{proof}

Our main result concerning the relationship of Ramsey degrees and Ramsey degrees for embeddings is the following.

\begin{theorem}\label{thm:mystery}
 Let $\K$ be a relational Fra\"iss\'e class which is Ramsey. Then the 
Ramsey degree for embeddings of $\K$ is infinite or a finite power of 2.
\end{theorem}

We refer to Examples~\ref{exas:rationals} for some natural 
examples of relational Fra\"iss\'e classes which are Ramsey and have infinite Ramsey degree for embeddings.
We prove Theorem~\ref{thm:mystery} in Section~\ref{sec:cor2}.

%\begin{theorem}\label{thm:mystery} Let $\K$ be a relational Fra\"iss\'e class such that $\Flim(\K)$ is 
% $\omega$-categorical. If $\K$ is Ramsey, then its Ramsey degree for embeddings is a power of 2.
%\end{theorem}
%
%We do not know of a `direct' proof of this result. We obtain Theorem~\ref{thm:mystery} in the next section via some theory of order forgetful expansions.

%%%%%%%%%%%%%%%%%%%%%%%%%%%%%%%%%%%%%%%%%%

\subsection{Proof of Theorem~\ref{thm:char}} \label{sec:proof}
Theorem~\ref{thm:char} is a consequence of the following 
two propositions which in fact establish something stronger.

We say that a class of finite structures $\mathcal D$ is {\em cofinal in} another 
such class $\K$ if for all $A\in\K$ there exists $B\in\mathcal D$ such that $A\le B$.

\begin{proposition}\label{prop:ramtoam} Let $d\in\N$ and $\K$ be a Fra\"iss\'e class.
Assume that the class of structures with Ramsey degree  for embeddings  at most $d$ in $\K$ is cofinal in $\K$.
Then $M(G_\K)$ has size at most $d$.
\end{proposition}

\begin{proof}
Write $G:=G_\K$ and $F:=\Flim(\K)$. Let $A\in\K, a_0\in F^A$ and write $A_0:=\im(a_0)$. 
Consider the map $\Phi: G\rightarrow F^A$, $g\mapsto g\circ a_0$ . 
By homogeneity of $F$, $\Phi$ is surjective. We have for all $g,h\in G$
$$g\circ a_0=h\circ a_0\Longleftrightarrow gG_{(A_0)}=hG_{(A_0)}.$$
Hence, $\Phi$ induces a bijection $\emb$ from $G/G_{(A_0)}$ onto $F^A$. Observe that 
\begin{equation}\label{eq:e}
g\circ \emb(hG_{(A_0)})=g\circ (h \circ a_0)= (gh)\circ a_0=\emb((gh)G_{(A_0)}). 
\end{equation}

\noindent{\em Claim 1.} Assume that $A$ has Ramsey degree  for embeddings at most $d$ in $\K$. 
Let $k\in\N$ and $f:G\to [k]$ be constant on each $gG_{(A_0)}\subseteq G$ for $g\in G$. Then, for every finite $H\subseteq G$ there exists $g\in G$ such that $|f(gH)|\le d$.\medskip

\noindent{\em Proof of Claim 1:} The function $f$ induces a function $\tilde f$ from 
$G/G_{(A_0)}$ to $[k]$. Note that $\tilde f\circ \emb^{-1}:F^A\to [k]$. 
There is a finite substructure $B\subseteq F$ such that  
\begin{equation}\label{eq:B}
\{\emb(hG_{(A_0)})\mid h\in H\}\subseteq B^A.
\end{equation}
By Lemma~\ref{lem:infinite2} there is $b\in F^B$ such that $|(\tilde f\circ \emb^{-1})(b\circ B^A)|\le d$.
By homogeneity of~$F$ there is a $g\in G$ such that $g\circ \textup{id}_B=b$.
We show that $g$ is as desired, namely $f(gh)\in (\tilde f\circ \emb^{-1})(b\circ B^A)$ for every $h\in H$: 
\begin{align*}
f(gh)=\tilde f((gh)G_{(A_0)})&=\tilde f\circ \emb^{-1} (\emb((gh)G_{(A_0)}))
=\tilde f\circ \emb^{-1} (g \circ \emb(hG_{(A_0)})  )
\end{align*}
where the last equality follows from \eqref{eq:e}. By \eqref{eq:B} we have 
$g\circ \emb(hG_{(A_0)})\in g\circ B^A= b\circ B^A$, and our claim follows.\hfill$\dashv$\medskip

For $n\in\N,n\ge 1,$ consider $\mathbb R^n$ with the Euclidian norm $\|\cdot\|$. For $\varepsilon>0$ and $x\in\mathbb R^n$ let $$B_{\varepsilon}(x) :=\{y\in \mathbb R^n\mid \|x-y\|<\varepsilon\}.$$

As a topological group $G$ carries its {\em left} uniformity, that is, the uniformity with basic entourages $\{(g,h)\mid g^{-1}h\in G_{(A)}\}$ for $A\in\Age(M), A\subseteq M$.
\medskip

\noindent{\em Claim 2.} Let $n$ be a positive integer, $f:G\to \mathbb R^n$ be left uniformly continuous and bounded, $H\subseteq G$ be finite and $\varepsilon$ be a positive real.
Then there are $g\in G$ and $h_0,\ldots,h_{d-1}\in H$ such that 
\begin{equation}\label{eq:claim2}
\textstyle
f(gH)\subseteq \bigcup_{\nu<d}B_{\varepsilon}(f(gh_\nu)).
\end{equation}

\noindent{\em Proof of Claim 2:} By left uniform continuity of $f$ there is a finite substructure  
$A'\subseteq F$ such that $\|f(g)- f(g')\|<\varepsilon/6$ for all $g,g'\in G$ with $gG_{(A')}=g'G_{(A')}$. 
By our cofinality assumption, there exist $A''\in\K$ and $a'\in (A'')^{A'}$ such that $A''$ 
has Ramsey degree for embeddings at most $d$ in $\K$. Since $F$ is homogeneous, there is an embedding $a''\in F^{A''}$ such that $a''\circ a'=\id_{A'}$. Hence, the image $A$ of $a''$ has Ramsey degree for embeddings at most $d$ in $\K$, and $A'\subseteq A\subseteq F$. Thus $G_{(A)}\subseteq G_{(A')}$, so
for all $g,g'\in G$ with $gG_{(A)}=g'G_{(A)}$
\begin{equation}\label{eq:A}
\|f(g)- f(g')\|<\varepsilon/6.
\end{equation}

We claim that there exists a function $\tilde f:G\to \mathbb R^n$ such that
\begin{enumerate}\itemsep=0pt
% \item[(a)] $\im(\tilde f)\subseteq\im(f)$;
\item[(a)] $\im(\tilde f)$ is finite;
\item[(b)] $\tilde f$ is constant on $gG_{(A)}$ for every $g\in G$;
\item[(c)] $\|f(g)-\tilde f(g)\|<\varepsilon/2$ for every $g\in G$.
\end{enumerate}
By (a) and (b) we can apply Claim 1 and obtain some $g\in G$ such that $|\tilde f(gH)|\le d$. Choose $h_0,\ldots, h_{d-1}\in H$
such that $\tilde f(gH)=\{\tilde f(gh_\nu) \mid \nu<d\}$. To verify~\eqref{eq:claim2}, let $h\in H$ be given.
We have to show that there exists $\nu<d$ such that  
$\|f(gh)-f(gh_\nu)\|<\varepsilon$. Indeed, this holds for $\nu<d$ such that $\tilde f(gh)=\tilde f(gh_\nu)$, because by (c) 
we have both $\|f(gh)-\tilde f(gh_\nu)\|=\|f(gh)-\tilde f(gh)\|<\varepsilon/2$ and
$\|\tilde f(gh_\nu)-f(gh_\nu)\|<\varepsilon/2$.

Thus, we are left to find $\tilde f$ with properties (a)-(c).

As $f$ is bounded, its image is contained in a compact subset of $\mathbb R^n$. Choose finitely many points
$y_\nu\in\mathbb R^n, \nu<k',$ such that this compact set is covered by $\bigcup_{\nu<k'} B_{\varepsilon/6}(y_\nu)$. 
Assume that precisely the first $k\le k'$ balls 
$B_{\varepsilon/6}(y_\nu)$ contain a point from the image of $f$. For $\nu<k$ choose $\widehat{\nu}\in G$ such 
that $f(\widehat{\nu})\in B_{\varepsilon/6}(y_\nu)$. Then 
$\bigcup_{\nu<k}B_{\varepsilon/3}(f(\widehat{\nu}))$ covers the image of~$f$. Hence, for every $g\in G$ we can choose
$\nu_g<k$ such that 
\begin{equation}\label{eq:nug}
\|f(g)-f(\widehat{\nu_g})\|<\varepsilon/3. 
\end{equation}
Let $c:G\to G$ be a selector for the partition $\{gG_{(A)}\mid g\in G\}$ of $G$, that is, for all $g,g'\in G$
 we have $c(g)\in gG_{(A)}$, and $c(g)=c(g')$ if and only if $gG_{(A)}=g'G_{(A)}$.
Define 
$$
\tilde f(g):=f(\widehat{\nu_{c(g)}}).
$$ 
Then $\tilde f$ satisfies (a) and (b). 
For all $g\in G$ we have $c(g)\in gG_{(A)}$, so $gG_{(A)}=c(g)G_{(A)}$ and thus 
$\|f(g)-f(c(g))\|<\varepsilon/6$ by \eqref{eq:A}. As
$\|f(c(g))-f(\widehat{\nu_{c(g)}})\|<\varepsilon/3$ by \eqref{eq:nug}, we conclude 
that $\tilde f$ satisfies  (c).\hfill$\dashv$\medskip

We aim to show that every $G$-flow has an orbit of size at most $d$ (Lemma~\ref{lem:orb}). So let $X$ be a $G$-flow. We are looking for some $x_0\in X$ such that
\begin{equation}\label{eq:goal}
|G\cdot x_0|\le d.
\end{equation}

Recall that the compact Hausdorff space $X$ carries a unique uniformity compatible with its topology. 
Suppose $f$ is a uniformly continuous function from $X$ into $\mathbb R^{n}$ for some $n\ge 1$. For each $x\in X$
define the function $f_x:G\to  \mathbb R^{n}$ by 
$$
f_x(g):=f(g^{-1}\cdot x).
$$ 
Then $f_x$ is left uniformly continuous. This follows from the well-known fact that
for every $x\in X$ the map $g\mapsto g^{-1}\cdot x$ is left uniformly continuous (see~e.g.~\cite[L2.1.5]{pestov}). 

%\noindent{\em Claim 3.} Let $Y$ be a uniform space and assume $f:X\to Y$ is uniformly continuous. For $x\in X$ define the function $f_x:G\to Y$ by $f_x(g):=f(g^{-1}\cdot x)$. Then $f_x$ is uniformly continuous.\medskip
%
%\noindent{\em Proof of Claim 3:} It suffices to show that the map $g\mapsto g^{-1}x$ from $G$ to $X$ is left uniformly continuous. Let $\mathcal U$ Given $V\in \mathcal U$ we have to find a neighborhood $U$ of the identity in $G$ such that $(g^{-1}x,h^{-1}x)\in V$ whenever $g^{-1}h\in U$. Let $V'\in \mathcal U$ be symmetric and such that $V'V':=\{(x,y)\mid \exists z\in X:(x,z),(z,y)\in V'\}\subseteq V$.For $y\in X$ choose a neighborhood $U_y$ of the identity in $G$ and a neighborhood $W_y$ of $y$ in $X$ such that $U_y\cdot W_y\subseteq V'(y):=\{z\in X\mid (y,z)\in V'\}$; such neighborhoods exist by continuity of the action. By compactness of $X$ there are $k\in\N$ and $y_\nu\in X,\nu<k,$ such that $X=\bigcup_{\nu<k}W_{y_\nu}$. Then $U:=\bigcap_{\nu<k}U_{y_\nu}$ is as desired. Given $x\in X$ choose $\nu<k$ such that $h^{-1}x\in W_{y_{\nu}}$. If $g^{-1}h\in U$, then $g^{-1}h\in U_{y_{\nu}}$, so $(g^{-1}h)(h^{-1}x)=g^{-1}x\in V'(y_{\nu})$. Hence $(y_{\nu},g^{-1}x)\in V'$ and $(y_{\nu},h^{-1}x)\in V'$. By choice of $V'$ we get $(g^{-1}x,h^{-1}x)\in V$ as claimed.\hfill$\dashv$\medskip

With a triple $(H,f,\varepsilon)$ for a finite subset $H\subseteq G$, and a bounded, 
uniformly continuous function $f:X\to\mathbb R^{n}$, and a real $\varepsilon>0$ we associate the set
$$\textstyle
Y(H,f,\varepsilon):=\left\{x\in X\mid \exists h_0,\ldots, h_{d-1} \in H: f_x(H)\subseteq \bigcup_{\nu<d}\overline{B_{\varepsilon}(f_x(h_\nu))}\right\}.
$$
Since $H$ is finite, $Y(H,f,\varepsilon)$ is a finite union of closed sets of the form
$\{x\in X\mid f_x(H)\subseteq C\}$ for $C\subseteq \mathbb R^n$ closed, and consequently, 
$Y(H,f,\varepsilon)$ is closed.

%\noindent{\em Claim 4.} $Y(H,f,\varepsilon)$ is closed.\medskip
%
%\noindent{\em Proof of Claim 4:} Since $H$ is finite, $Y(H,f,\varepsilon)$ is a finite union of sets of the form
%$Z(H,f,C):=\{x\in X\mid f_x(H)\subseteq C\}$ for $C\subseteq \mathbb R^n$ closed. It 
%thus suffices to show that each such set $Z(H,f,C)$ is closed. We verify that its complement is open. If $y\notin Z(H,f,C)$, then there is $h\in H$ such that $f(h^{-1}y)\notin C$. Since $C$ is closed there exists $\delta>0$ such that $B_\delta(f(h^{-1}y))$ is disjoint from $C$. Since $f$ is continuous we find a neighborhood $U$ of $h^{-1}y$ in $X$such that $f(U)\subseteq B_\delta(f(h^{-1}y))$. Since the action is continuous we find neighborhoods $E$ of $h^{-1}$ in $G$ and $U'$ of $y$ in $X$ such that $E\cdot U'\subseteq U$. In particular,$h^{-1}\cdot U'\subseteq U$, so $f(h^{-1}\cdot U')\subseteq B_\delta(f(h^{-1}y)$ is disjoint from $C$. Hence, $f_{y'}(h)\notin C$ for all $y'\in U'$ and $U'$ is a neighborhood of $y$ disjoint from $Z(H,f,C)$. \hfill$\dashv$\medskip

\medskip

\noindent{\em Claim 3.} The family of closed sets $Y(H,f,\varepsilon)$ with $H,f,\varepsilon$ as above has the finite intersection property.\medskip

\noindent{\em Proof of Claim 3:}
For $j<\ell$ let $H_j\subseteq G$ be finite, $\varepsilon_j>0$ and $f^j:X\to \mathbb R^{n_j}$ for $n_j\ge 1$.
Set $H:=\bigcup_{j<\ell}H_j,\varepsilon:=\min_{j<\ell}\varepsilon_j, n:=\sum_{j<\ell}n_j$ and
define $f:X\to\mathbb R^n$ by $f(x):=f^0(x)*\cdots *f^{\ell-1}(x)$ where $*$ denotes concatenation. Then $f$ is uniformly continuous and bounded.

Let $x\in X$ be arbitrary. Since $f_x:G\to\mathbb R^n$ is left uniformly continuous, Claim~2 applies, and there exist
$g\in G$ and $h_0,\ldots,h_{d-1}\in H$ such that
$f_x(gH)\subseteq \bigcup_{\nu<d}B_{\varepsilon}(f_x(gh_\nu))$. In other words,
% for all $h\in H$ there exists $\nu<d$ such that
\begin{equation}\label{eq:cl41}
\forall h\in H\ \exists \nu<d:\ f(h^{-1}g^{-1}x)\in B_{\varepsilon}(f(h_\nu^{-1}g^{-1}x)).
\end{equation}

Any $y\in \mathbb R^n$ can be written as $y[0]*\cdots *y[\ell-1]$, where $y[j]\in\mathbb R^{n_j}$ for all 
$j<\ell$. In this notation, $f_x(g)[j]=f^j_x(g)$ for all $g\in G,x\in X, j<\ell$.
Clearly, $f_x(g)\in B_\varepsilon(y)$ implies $f_x(g)[j]\in B_\varepsilon(y[j])$ for all $y\in \mathbb R^n, j<\ell$.
Writing $x_0:=g^{-1}x$, ~\eqref{eq:cl41} yields: 
\begin{equation*}
\forall j<\ell\ \forall h\in H_j\ \exists \nu<d:\ f(h^{-1}g^{-1}x)[j]=f^j_{x_0}(h)\in B_{\varepsilon}(f^j_{x_0}(h_\nu)).
\end{equation*}
Since $\varepsilon\le \varepsilon_j$ we obtain
$$\textstyle
\forall j<\ell:\ f^j_{x_0}(H_j)\subseteq\bigcup_{\nu<d}B_{\varepsilon_j}(f^j_{x_0}(h_\nu)).
$$
Thus,  $x_0\in \bigcap_{j<\ell}Y(H_j,f^j,\varepsilon_j)\neq\emptyset$.\hfill$\dashv$\medskip

By Claim 3 and since $X$ is compact, there exists an $x_0$ in the intersection of all the sets $Y(H,f,\varepsilon)$, 
$(H,f,\varepsilon)$ a triple as above. We claim that $x_0$ satisfies \eqref{eq:goal}.
Assume otherwise that there  are $g_0,\ldots,g_{d}\in G$ such that $g_0x_0,\ldots,g_{d}x_0$ are pairwise distinct. Choose 
 $f:X\to[0,1]\subseteq \mathbb R^1$ uniformly continuous such that $f(g_\nu x_0)=\nu/d$ for all 
$\nu\le d$. 
Then $x_0\notin Y(\{g^{-1}_\nu\mid \nu\le d\}, f,\varepsilon)$ for a small enough $\varepsilon>0$, a contradiction.\end{proof}

\begin{proposition} \label{prop:amtoram}
Let $d\in\N$, $F$ be countable and locally finite, $G:=\Aut(F)$ and $A\in\Age(F)$ such that
$F$ is $A$-homogeneous. If $M(G)$ has size at most $d$, then $F\hookrightarrow (B)^A_{k,d}$ for all $B\in\Age(F)$ and $k\ge 2$.
\end{proposition}

\begin{proof} Assume that $|M(G)|\le d$, and let $B\in\Age(F),k\ge 2$ and $\chi_0:F^A\to [k]$ be a colouring. 
Note that $[k]^{F^A}$ is compact Hausdorff in the product topology with $[k]$ being discrete. 
The group $G$  acts continuously on $[k]^{F^A}$ by shift
$(g,\chi)\mapsto g\cdot \chi$, where $g\cdot\chi$ colours $a\in F^A$ by $\chi(g^{-1}\circ a)$. 
Consider the orbit closure $\overline{G\cdot \chi_0}$ of $\chi_0$. 
By Lemma~\ref{lem:orb}, 
the induced action of $G$ on $\overline{G\cdot \chi_0}$ has an orbit of size at most $d$, that is,
there exist $\chi_1\in \overline{G\cdot \chi_0}$ and $\psi_0,\ldots,\psi_{d-1}\in \overline{G\cdot \chi_0}$ such that
$G\cdot \chi_1=\{\psi_i\mid i<d\}$. 

Let $b\in F^B$. Observe that $b\circ B^A$ is a finite subset of $F^A$.
Since $\chi_1\in \overline{G\cdot \chi_0}$, there exists a 
$g\in G$ such that $g\cdot \chi_0$ and $\chi_1$ agree on $b\circ B^A$. Note that $g^{-1}\circ b\in F^{B}$, so we are left to show that $|\chi_0(g^{-1}\circ b\circ  B^A)|\le d$. We fix some $a_0\in F^A$, and claim that
for all $a\in g^{-1}\circ b\circ  B^A$ there exists a $\nu<d$ such that $\chi_0(a)=\psi_\nu(a_0)$. 
To see this, let $a\in g^{-1}\circ b\circ B^A\subseteq F^A$ and choose $h\in G$ 
such that $h\circ a_0=a$. Such an $h$ exists since $F$ is $A$-homogeneous. Then
$$
\chi_0(a)= (g\cdot\chi_0)(g\circ a)=\chi_1(g\circ a)= \chi_1((gh)\circ a_0) =((gh)^{-1}\cdot \chi_1)(a_0),
$$ 
where the second equality follows from  $g\circ a \in b\circ B^A$ and the choice of $g$.
As $(gh)^{-1}\cdot \chi_1\in G\cdot \chi_1$, and by choice of $\chi_1$, there exists $\nu<d$ 
such that $(gh)^{-1}\cdot \chi_1=\psi_\nu$. Thus 
$\chi_0(a)=\psi_\nu(a_0)$ as claimed.\end{proof}

\begin{proof}[Proof of Theorem~\ref{thm:char}.]
(1) $\Rightarrow$ (2). Write $F=\Flim(\K)$ and let $A\in\K=\Age(F)$. Then $F$ and $A$ satisfy the assumptions of Proposition~\ref{prop:amtoram}, so $F\hookrightarrow(B)^A_{k,d}$ for all $B\in\K$ and $k\ge 2$. Now apply Lemma~\ref{lem:infinite2}.

(2) $\Rightarrow$ (1). By Proposition~\ref{prop:ramtoam}.
\end{proof}

\subsection{Corollaries}\label{sec:cor}

\begin{corollary}\label{cor:asymptotic} Let $d\in\N$ and  $\K$ be a Fra\"iss\'e class.
The following are equivalent.
\begin{enumerate}\itemsep=0pt
 \item The class of structures with Ramsey degree for embeddings at most $d$ in $\K$ is cofinal in $\K$.
\item $\K$ has Ramsey degree for embeddings at most $d$.
\end{enumerate}
%In particular, $\K$ is Ramsey if and only if every $A\in\K$ embeds into a rigid structure from $\K$ which has ramsey degree 1 in $\K$.
\end{corollary}

\begin{proof}
Assume (1). By Proposition~\ref{prop:ramtoam} we have $|M(G_\K)|\le d$. 
As $F:=\Flim(\K)$ is Fra\"iss\'e, Proposition~\ref{prop:amtoram} implies $F\hookrightarrow (B)^A_{k,d}$ for all
$A,B\in\K$. Then Lemma~\ref{lem:infinite2} implies~(2).
\end{proof}

It is noted in \cite[$\S 1$(D)]{kpt} that a separable metrizable group $G$ is extremely amenable, i.e.
$M(G)$ has size 1, if and only if every metrizable $G$-flow has a fixed point. In this context it might be of interest to note:

\begin{corollary}  Let $d\in\N$ and $\K$ be a Fra\"iss\'e class.
The following are equivalent.
\begin{enumerate}\itemsep=0pt
\item $M(G_\K)$ has size at most $d$.
\item Every continuous action of $G_\K$ on the Cantor space has an orbit of size at most $d$.
\end{enumerate}
\end{corollary}

\begin{proof}
(1) implies (2)  by Lemma~\ref{lem:orb}. Conversely, assume (2).
Let $A\in\K$ be arbitary and write $F:=\Flim(\K)$. Then $F$ and $A$ satisfy the assumptions of
Proposition~\ref{prop:amtoram}. In the proof of this proposition 
we only require the following for $G_\K$: for all $k\ge 2$
and all $\chi_0\in[k]^{F^A}$, the shift action of $G_\K$ restricted to $\overline{G_\K\cdot\chi_0}$ has a small orbit.
But $[k]^{F^A}$ is homeomorphic to the Cantor space and the restricted shift is a continuous action on this space.
Thus (2) suffices to carry out this proof and we conclude that $F\hookrightarrow(B)^A_{k,d}$ for all $B\in\K=\Age(F)$. 
By Lemma~\ref{lem:infinite2} 
every $A\in\K$ has Ramsey degree for embeddings at most~$d$ in~$\K$. Then 
Proposition~\ref{prop:ramtoam} implies~(1).
\end{proof}

%Finally, we point out the following strengthening of Theorem~\ref{thm:extremeamenability} due to Nguyen van Th\'e \cite[T1]{nguyen}. As already said, our proof of Theorem~\ref{thm:char} follows similar lines as the argument given there.
%
%\begin{corollary}  Let $\K$ be a Fra\"iss\'e class. Then $G_\K$ is extremely amenable if and only if $\K$ is rigid and Ramsey.
%\end{corollary}
%\proof If $G_\K$ is extremely amenable, then $|M(G_\K)|=1$. Theorem~\ref{thm:char} and Corollary~\ref{cor:asymptotic} imply that $\K$ has Ramsey degree for embeddings 1. Proposition~\ref{prop:emb} implies that $\K$ is Ramsey and rigid.
%
%Conversely, if $\K$ is Ramsey and rigid, then, by Proposition~\ref{prop:emb}, the Ramsey degree for embeddings of $\K$ is 1. Theorem~\ref{thm:char} then implies $|M(G_\K)|=1$, that is, $G_\K$ is extremely amenable.
%\qed

\section{Measure concentration}

We say that a probability measure is {\em concentrated on} any set of measure 1.  In this section we prove the following.

\begin{theorem}\label{thm:concentr}
%Assume that $L$ is relational. Let $\K$ be a Ramsey class in the language $L$. 
Let $\K$ be a relational Fra\"iss\'e class which is Ramsey.
If $G_\K$ is amenable, then it is uniquely ergodic and the (unique) $G_\K$-invariant Borel 
probability measure on $M(G_\K)$ is concentrated on a (unique) dense $G_\delta$ orbit.
\end{theorem}

In Section~\ref{sec:forget} we construct a forgetful order expansion using the Ramsey property, in Section~\ref{sec:concentr} 
we prove Theorem~\ref{thm:concentr}, and the final Section~\ref{sec:cor2} contains some observations 
concerning the $\omega$-categorical case and a proof of (a stronger version of) Theorem~\ref{thm:mystery}.

\subsection{Forgetful order expansions}\label{sec:forget}

An ordered Fra\"iss\'e class $\K^*$ in the language $L\cup\{<\}$ is called {\em forgetful} if for all $A,B\in\K$ and $\K^*$-admissible 
orderings $<^A, <^B$ on $A,B$ respectively, $(A,<^A)\cong (B,<^B)$ whenever $A\cong B$; here $\K$ 
denotes the $L$-reduct of $\K^*$.

For example, the orderings of $B_\infty$ and $V_{\infty,F}$ mentioned in Example~\ref{exas} have forgetful
 ages (see~\cite[$\S6$]{kpt} for details). The following is easy to see (cf.~\cite[P5.6]{kpt}).

\begin{lemma}\label{lem:forgetramsey} Let $\K^*$ be a forgetful ordered Fra\"iss\'e class in the language $L\cup\{<\}$ and $\K$ 
its $L$-reduct. Then $\K^*$ has the ordering property, and  $\K^*$ is Ramsey if and only if so is $\K$.
\end{lemma}

Before showing that the Ramsey property ensures the existence of reasonable forgetful expansions, we present a well-known technical lemma. Informally, this technical lemma guarantees a monochromatic copy of a given $B$ when copies of several different $A_i$ are coloured simultaneously.

\begin{lemma}\label{lem:ramseyref}
Let $\K$ be a Ramsey class. Let $n\in \mathbb{N}$, $k_0, \ldots, k_{n-1}\in\mathbb{N}$, $A_0, \ldots, A_{n-1} , B\in \K$. 
Then there exists a $C\in \K$ with the following property: for any family of
colourings $\chi_i: \binom{C}{A_i}\rightarrow [k_i]$, $i\in [n]$, there
 exists a $B'\in \binom{C}{B}$ such that 
%$\chi_i\upharpoonright_{B'}$
$\chi_i\upharpoonright\binom{B'}{A_i}$ 
is constant for all $i\in [n]$.
\end{lemma}
\begin{proof}
Let $C_0:= B$, and for every $0<i\le n$ choose $C_i\in\K$ such that
$C_i\to (C_{i-1})_{k_{i-1},1}^{A_{i-1}}$. Let $C:= C_n$. Then by using the definition of the $C_i$ and a 
straightforward induction on $j\in [n]$ we obtain that there is a $C'_{n-1-j}\in \binom{C}{C_{n-1-j}}$ such that 
%$\chi_i\upharpoonright_{C'_{n-1-j}}$ 
$\chi_{n-1-j}\upharpoonright \binom{C'_{n-1-j}}{A_{n-1-j}}$ 
is constant for all $i\in [n]\setminus [n-1-j]$. Setting $j= n-1$ yields $B'$ as in the statement. 
\end{proof}

\begin{lemma}\label{lem:forget}
Let $\K$ be a Fra\"{\i}ss\'{e} class in the language $L$. If $\K$ is Ramsey, then there exists a 
reasonable, forgetful ordered Fra\"{\i}ss\'{e} class $\K^*$ in the language $L\cup \{<\}$ with $L$-reduct $\K$.
\end{lemma}
\begin{proof} Let $F:=\Flim(\K)$ and consider the space {\em LO} of linear orders on $F$ (cf.~Example~\ref{exa:lo}). 
Let $(A,B)$ range over pairs with $A\in\K$ and $B\subseteq F$.
%with $A\subseteq B\subseteq F$. 
Call $R\in\textit{LO}$ order forgetful for $(A,B)$ if 
$(A',R\upharpoonright A')\cong(A'',R\upharpoonright A'')$ for all $A',A''\in{B\choose A}$. 

\medskip

\noindent{\em Claim.} If $n\ge 1$ and $(A_0,B_0), \ldots,(A_{n-1},B_{n-1})$ are pairs as above with all $B_i\subseteq F$ finite, 
then there exists $R\in\textit{LO}$ 
that is order forgetful for every $(A_i,B_i), i\in [n]$.
%For every $A_0, \ldots A_{n-1}\in\K$ and $B\subseteq F$ finite there exists $R\in\textit{LO}$ that is order forgetful for $(A_i,B)$ for all $i\in [n]$.
\medskip

\noindent{\em Proof of Claim:} Choose $B\subseteq F$ finite 
such that $\bigcup_{i\in [n]}B_i\subseteq B$. It suffices to find
an order which is order forgetful for every $(A_i, B), i\in [n]$. Fix an arbitrary order $R\in\textit{LO}$. 
For $i\in [n]$ let $\chi_i$ colour each  $A_i'\in{F \choose A_i}$ by the isomorphism 
type of $(A_i',R\upharpoonright A_i')$, and let $k_i\in\N$ be the number of colours 
of $\chi_i$. 
%Inductively for every $0<i\le n$ choose $C_i\in\K$ such that
%$C_i\to (C_{i-1})_{k_i}^{A_i}$. 
%Fix $c\in F^{C_n}$. Now choose 
%copies $C'_i$ of $C_i$ with $C'_0\subseteq C'_1\subseteq\cdots\subseteq C_n'=C_n$ as follows. 
%For each $0<i\le n$ let $\chi'_i$ be the colouring mapping $A_i'\in{C'_i\choose A_i}$ to 
%$\chi_i(c(A_i'))$; set $C'_n:=C_n$ and inductively for $0<i\le n$ choose $C'_{n-i}\in {C'_{n-i+1}\choose C_{n-i}}$
%such that  $\chi'_{n-i+1}$ is monochromatic on ${C'_{n-i}\choose A_{n-i+1}}$.
By Lemma~\ref{lem:ramseyref} and homogeneity of $F$ there exist $B'\subseteq F$ and $g\in \Aut(F)$ such that 
%$B'$ is monochromatic and 
$g(B')=B$ and each $\chi_i$ is constant on $\binom{B'}{A_i}$. 
By definition of the $\chi_i$ this means that $R$ is order 
forgetful for $(A_i, B')$ for all $i\in [n]$. Hence, $g(R)$ is order forgetful for all $(A_i, B), i\in [n]$.
\hfill$\dashv$\medskip

For every $A\in\K$ and $B\subseteq F$ finite, the set of orders that are order forgetful for $(A,B)$ is 
closed in $\textit{LO}$. By the claim and compactness, there exists $R\in\textit{LO}$ which is order forgetful 
for all pairs $(A,B)$ such that $A\in\K$ and $B\subseteq F$ is finite.
Then $R$ is order forgetful for $(A,F)$ for every $A\in\K$. Equivalently, $\K^*:=\Age(F,R)$ is forgetful. 
To see that $\K^*$ is Fra\"iss\'e, observe that $\K^*$ is hereditary and has joint embedding. As $\K^*$ is 
Ramsey by Lemma~\ref{lem:forgetramsey}, it has amalgamation by Lemma~\ref{lem:Ramam} (and Remark~\ref{rem:rigiddegrees}; 
note that $\K^*$ is rigid because it is ordered). According to Lemma~\ref{lem:reason}, $\K^*$ is reasonable.
\end{proof}

\begin{examples}\label{exas:rationals}

The structures $F_1:=(\mathbb{Q}, \Betw), F_2:=(\mathbb{Q}, \Cycl), F_3:=(\mathbb{Q}, \Sep)$ and 
$F_4:=(\mathbb{Q}, =)$ are Ramsey (see \cite{cameron} for definitions). 
If $<$ is the rational order, then $\K_i^*:=\Age((F_i,<))$ is forgetful with reduct $\K_i:=\Age(F_i)$.
By Lemmas~\ref{lem:reason},~\ref{lem:forgetramsey} and Theorem~\ref{thm:minflow}, $M(G_{\K_i})$ is 
$\overline{G_{\K_i}\cdot <}$. Then $M(G_{\K_1})$
is the 2-element discrete space, Hence, by Theorem~\ref{thm:char}, 
$\K_1$ has Ramsey degree for embeddings 2. 
Theorem~\ref{thm:minflow} also allows to 
explicitly describe $M(G_{\K_i})$ for $i=2,3,4$ and these have the size of the continuum. 
 Hence, $\K_2,\K_3$ and $\K_4$ have infinite Ramsey degree for embeddings.
%
%%%OLD TEXT
%The structures $(\mathbb{Q}, \Betw), (\mathbb{Q}, \Cycl), (\mathbb{Q}, \Sep), (\mathbb{Q}, =)$ are Ramsey (see \cite{cameron} for the definition of the relations $\Betw$, $\Cycl$ and $\Sep$). If $<$ denotes the total order of the rationals, and $F$ is any of these four structures, then $\K^*= \Age(F, <)$ is a reasonable, forgetful order expansion of $\K=\Age(F)$ with the ordering property.  By Theorem~\ref{thm:minflow} we have that $M(G_\K)$ is the 2-element discrete space if $F=(\mathbb{Q}, \Betw)$,  thus $\Age(\mathbb{Q}, \Betw)$ has Ramsey degree for embeddings 2 according to Theorem~\ref{thm:char}.  Also by Theorem~\ref{thm:minflow} one can explicitly describe the universal minimal flow in the remaining three cases,  and those have the size of the continuum.  Hence, those classes have infinite Ramsey degree for embeddings.
\end{examples}

\begin{examples}
Let $\K$ be a Fra\"iss\'e class of digraphs such that there is a directed cycle in~$\K$. 
%Then $\K$ does not have a forgetful order expansion $K^*$, as it would violate transitivity of the given total order in $K^*$. 
Then there does not exist a forgetful ordered Fra\"iss\'e class with $L$-reduct $\K$:
by forgetfullness, every directed edge in any $A\in \K$ would be ordered in the same way and then
a directed cycle contradicts transitivity of the order. 
For example, this applies to the age of the universal homogeneous digraph, 
the random tournament and the local order (see~\cite{mcph}). 
\end{examples}

\subsection{Proof of Theorem~\ref{thm:concentr}}\label{sec:concentr}

Let $F:=\Flim(\K)$ and $L$ denote the relational language of~$\K$. Since $\K$ is assumed to be Ramsey, Lemma~\ref{lem:forget} applies and
there is a reasonable forgetful ordered Fra\"iss\'e class $\K^*$ in the language $L\cup\{<\}$
with $L$-reduct $\K$. By Lemma~\ref{lem:forgetramsey} and Theorem~\ref{thm:minflow}, 
 $\Flim(\K^*)=(F,<^*)$ for some order~$<^*$, and $X_{\K^*}=\overline{G_\K\cdot <^*}$ is the universal minimal flow of~$G_\K$.

Assume that $G_{\K}$ is amenable. In order to verify that $G_{\K}$ is uniquely ergodic, it 
suffices by Proposition~\ref{prop:randomorder} to show that for every consistent random ordering $(R_A)_{A\in\K}$ we
 have that each random variable $R_A$ is uniformly distributed. By forgetfulness, for any two $\K^*$-admissible
 orderings $<,<'$ on $A$ there is an $\alpha\in\Aut(A)$ such that $\alpha(<)=<'$, and 
then $\Pr[R_A=<']=\Pr[\alpha^{-1}\circ R_A=<]=\Pr[R_A=<]$ where the latter equality follows from
$(R_A)_{A}$ being consistent. 
%from the definition of consistency. 

Let $\mu$ denote the unique $G_\K$-invariant Borel probability measure  on $X_{\K^*}$. Recall 
the notation $U(<)$ from Proposition~\ref{prop:randomorder}. By this result, $U(<)$ and $U(<')$ 
have the same $\mu$-measure whenever $<$ and $<'$ are  $\K^*$-admissible orderings of the same finite subset of~$F$.

An order $R\in X_{\K^*}$ 
%if and only if $\Age(F,R)=\Age(F,<^*)$ (see~\cite[???]{kpt}) and such $R$ 
is outside $G_\K\cdot <^*$ if and only if $(F,R)\not\cong(F,<^*)$, if 
and only if $(F,R)$ is not homogeneous (cf.~Section~\ref{sec:fraisse}), if and only if there exist a finite $A\subseteq F$, some $(B,<^B)\in\K^*$ and
$a\in (B,<^B)^{(A,<^*\upharpoonright A)}$ such that $R$ is {\em bad for} $(B,<^B,a)$, meaning that 
there is no $b\in (F,R)^{(B,<^B)}$ with $b\circ a= \id_A$. 
As the language of $F$ is relational, we may assume that $B=\im(a)\cup\{p\}$ with $p\in F\setminus \im(a)$. 

Observe that the set of orders $R\in X_{\K^*}$ which are bad for $(B,<^B,a)$
is closed. Hence, $X_{\K^*}\setminus G_\K\cdot <^*$ is  $F_\sigma$, so 
$G_\K\cdot <^*$ is a dense $G_\delta$ orbit in $X_{\K^*}$ (see also \cite[14.3]{akl}). Since $X_{\K^*}$ is a Baire space, $G_{\K}\cdot <^*$ is clearly unique with this property. We prove that $\mu(G_\K\cdot <^*)=1$.
It suffices to show that for each $(B,<^B,a)$ with $B=\im(a)\dot\cup\{p\}$ as above, the set
$\mathcal B:=\{R\in X_{\K^*}\mid  R \text{ is bad for }(B,<^B,a)\}$ has $\mu$-measure 0.

We construct a sequence $(\mathcal U_n)_{n\in\mathbb N}$ such that for all $n\in\mathbb N$
\begin{enumerate}
\item[(a)] $\mathcal U_n$ is a cover of $\mathcal B$, i.e.\ $\mathcal B\subseteq \bigcup\mathcal U_n$;
\item[(b)] every $U\in\mathcal U_n$ equals some $U(<')$ such that $<'\supseteq <^*\upharpoonright A$ is a $\K^*$-admissible order with 
$|\dom(<')|=|A|+n$; 
\item[(c)] $\textstyle\mu(\bigcup \mathcal U_{n+1})\leq \frac{|A|+n}{|A|+n+1}\cdot\mu(\bigcup\mathcal U_n).$
\end{enumerate}
Here, $\dom(<')$ is the set linearly orderd by $<'$; note that (b) implies $\dom(<')\supseteq A$.

This finishes the proof: by (a) and (c) we have for all $n\in\mathbb N$
$$\textstyle
\mu(\mathcal B)\le \mu(\bigcup\mathcal U_n)\le 
\prod_{m<n}\frac{|A|+m}{|A|+m+1}\cdot \mu(\bigcup\mathcal U_0)=\mu(\bigcup\mathcal U_0)\cdot \frac{|A|}{|A|+n}\to_n 0.
$$

 Set $\mathcal U_0:=\{U(<^*\upharpoonright A)\}$ and assume that 
$\mathcal U_n$ is already defined. It suffices to find for every $U(<')\in\mathcal U_n$
some $p'\notin\dom(<')$ and a family $(<_i)_{i\in I}$ such that
\begin{enumerate}
\item[(a')] $\bigcup_{i\in I}U(<_i)\cap\mathcal B=U(<')\cap\mathcal B$;
\item[(b')]  for every $i\in I$, $<_i\supseteq <'$  is a $\K^*$-admissible order with $\dom(<_i)=\dom(<')\cup\{p'\}$;
\item[(c')] $\mu(\bigcup_{i\in I}U(<_i))\le\frac{|A|+n}{|A|+n+1}\cdot\mu(U(<'))$.
\end{enumerate}

Write $A':=\dom(<')$, and choose $R\in \mathcal B\cap U(<')$. Since 
$R\in \overline{G_\K\cdot <^*}$ there is a $g\in G_\K$ such that 
\begin{equation}\label{eq:Rg}\textstyle
 g(<^*)\upharpoonright A'= R\upharpoonright A'=<'.
\end{equation}
In particular, 
$g(<^*)\upharpoonright A=<'\upharpoonright A=<^*\upharpoonright A$ and
$(A,<^*\upharpoonright A)$ is a substructure of $(F,g(<^*))$. Since
$(F,g(<^*))$ is isomorphic to $(F,<^*)$, it is homogeneous, so
there exists an embedding $b\in (F,g(<^*))^{(B,<^B)}$ with $b\circ a=\id_A$. 
We set $p':=b(p)$ and claim that $p'\notin A'$. Otherwise, $\im(b)\subseteq A'$, so 
$b\in (F,R)^{(B,<^B)}$ by~\eqref{eq:Rg}, and this contradicts $R$ being bad for $(B, <^*, a)$.

Let $<_0,\ldots,<_{s-1}$ list the $\K^*$-admissible orders on $A'\cup\{p'\}$ 
extending $<'$, and note that $s\le |A'|+1$. Let $I\subseteq[s]$ consist of those $i<s$ such that 
$U(<_i)\cap\mathcal B\neq\emptyset$. Then (a') and (b') follow, and
we are left to verify (c'). The sets $U(<_i),i<s,$ partition $U(<')$ and, as already noted, 
have pairwise equal $\mu$-probability, so $\mu(U(<_i))=\mu(U(<'))/s$. Thus
\begin{equation}\label{eq:I}\textstyle
\mu(\bigcup_{i\in I}U(<_i))=|I|/s\cdot\mu(U(<')). 
\end{equation}

There exists $i_0<s$ such that $<_{i_0}=g(<^*)\upharpoonright (A'\cup\{p'\})$. Since 
$b\in (F,g(<^*))^{(B,<^B)}$ has  $\im(b)\subseteq A'\cup\{p'\}$, 
we have that $b\in (F,S)^{(B,<^B)}$ for every $S\in U(<_{i_0})$. Hence, no such $S$ is bad for $(B,<^B,a)$, that is, 
$U(<_{i_0})\cap\mathcal B=\emptyset$, so $i_0\notin I$. Thus 
$|I|<s$. Since $|A'|=|A|+n$, we have $s\le |A|+n+1$, so $|I|/s\le (s-1)/s\le (|A|+n)/(|A|+n+1)$. Hence, (c') follows from~\eqref{eq:I}.
%\end{proof}

\begin{comment}
\begin{proof}[Proof of Theorem~\ref{thm:concentr}.]
By Lemma~\ref{lem:forget} there is a reasonable, forgetful, ordered Fra\"iss\'e class $\K^*$ with $L$-reduct $\K$.
By Lemma~\ref{lem:forgetramsey} $\K^*$ has the ordering property. Now apply
Proposition~\ref{prop:concentr}.
\end{proof}
\end{comment}

\subsection{The $\omega$-categorical case}\label{sec:cor2}
Of particular interest are Fra\"iss\'e classes
$\K$ which have an $\omega$-categorical Fra\"iss\'e 
limit $F:=\Flim(\K)$. By the theorem of Ryll-Nardzewski (see e.g.\ \cite[T4.3.1]{zieglerbuch}) this happens e.g.\ if 
the language $L$ of $\K$ is finite and relational (cf.~\cite[T4.4.7]{zieglerbuch}), and 
is equivalent to $G_\K$ being {\em oligomorphic}: 
for every $n\in\mathbb N$, $G_\K$ has only finitely 
many $n$-orbits. An {\em $n$-orbit of $G_\K$} is an orbit of the {\em diagonal action} of $G_\K$ on $F^n$ given by 
$g\cdot \bar a=g\cdot (a_0,\ldots, a_{n-1}):= g(\bar a)=(g(a_0),\ldots, g(a_{n-1}))$.

\begin{lemma}\label{lem:forgetoligo}
 Let $\K^*$ be a reasonable ordered Fra\"iss\'e class in the 
language $L\cup\{<\}$ with $L$-reduct $\K$. Then $G_{\K^*}$ is oligomorphic if and only if so is $G_\K$.
\end{lemma}

\begin{proof}
Let $F=\Flim(\K)$. By Theorem~\ref{thm:minflow} we have that $\Flim(\K^*)=(F, <^*)$ for 
some order~$<^*$ on $F$. As $G_{\K^*}$ is a subgroup of $G_\K$, it suffices to show that
 every orbit $T\subseteq F^n$ of~$G_\K$ that consists of tuples with all different entries 
is the union of finitely many $n$-orbits of~$G_{\K^*}$. Let $\bar s=(s_1, \ldots, s_n)$ and $\bar t=(t_1, \ldots, t_n)$
 be tuples in $T$ such that the unique extension of the partial isomorphism $s_1\mapsto t_1, \ldots, s_n\mapsto t_n$ to 
the substructures in~$F$ generated by $\bar s$ and $\bar t$ is a partial isomorphism of $F^*$. Then by homogeneity of $F^*$
 we have that $\bar s$ and $\bar t$ are in the same $n$-orbit of $G_{\K^*}$. As there are finitely many ways to define
 a ($\K^*$-admissible) order on the structure generated by a tuple in $T$, the claim follows.
\end{proof}

\begin{lemma}\label{lem:normalfinite} %Assume that $L$ is relational.
Let $\K^*$ be a reasonable ordered Fra\"iss\'e class in the 
language $L\cup\{<\}$ with $L$-reduct $\K$. Assume that $G_{\K^*}$ is oligomorphic.
If $G_{\K^*}$ is normal in~$G_\K$, then it has finite index in $G_\K$.
\end{lemma}

\begin{proof}
By reasonability $\Flim(\K^*)=(\Flim(\K),<^*)$ for some order $<^*$.
Consider the logic action of $G_{\K}$ on $\textit{LO}$ (Example~\ref{exa:lo}). Then $G_{\K^*}$ is the stabilizer of $<^*$.
% where $<^*$ is the order of $\Flim(\K^*)$.
Hence, $|G_{\K}:G_{\K^*}|=|G_\K\cdot <^*|$ and it suffices to show that $G_\K\cdot <^*$ is finite. 
If $G_{\K^*}$ is normal, then it fixes every $R\in G_\K\cdot <^*$. Thus every such $R$ is a union of 2-orbits.
As $G_{\K^*}$ is oligomorphic, there are only finitely many such $R$. 
\end{proof}

We use the following mode of speech from~\cite{akl}: let $\K$ be a Fra\"iss\'e class in the language~$L$; 
a {\em companion of $\K$} is a reasonable ordered Fra\"iss\'e class $\K^*$
in the language $L\cup\{<\}$ which is Ramsey, has the ordering property and has $L$-reduct $\K$.
Note:

\begin{proposition}\label{prop:ramseycompanion}
If a Fra\"iss\'e class is Ramsey, then it has a companion.
\end{proposition}

\begin{proof}
By Lemmas~\ref{lem:forget} and~\ref{lem:forgetramsey}. 
\end{proof}

\begin{proposition}\label{prop:boolean}
%Assume that $L$ is relational. Let $\K^*$ be a reasonable ordered Fra\"iss\'e class in the 
%language $L\cup\{<\}$ with the ordering property and 
%with $L$-reduct $\K$. Assume that $\K^*$ is Ramsey. 
%
Let $\K$ be a relational Fra\"iss\'e class that has a companion.
If $M(G_{\K})$ is finite, then $|M(G_{\K})|$ is a power of 2.
\end{proposition}

\begin{proof}
Let $L$ denote the relational language of $\K$ and let $\K^*$ be a companion of $\K$.
By Theorem~\ref{thm:minflow} we have that $F^*:=\Flim(\K^*)=(F,<^*)$ for $F:=\Flim(\K)$,  and that
$M(G_\K)$ is $\overline{G_\K\cdot <^*}$. Assume that $M(G_\K)$ is finite.
Then $G_\K\cdot <^*$ is finite, and since $G_{\K^*}$ is the stabilizer of $<^*$ in 
the logic action of $G_\K$ on $\textit{LO}$, 
$G_{\K^*}$ has finite index in $G_\K$. 
By Theorem~\ref{thm:extremeamenability}, $G_{\K^*}$ is extremely amenable. 
By Lemma~\ref{lem:todorpart}, $G_{\K^*}$ is normal in $G_\K$ and $|M(G_\K)|=|G_\K:G_{\K^*}|$.

Consider the diagonal actions of $G_\K$ and $G_{\K^*}$ on $F^2$. We claim that for every $g\in G_\K$ and 
every 2-orbit $S$ of $G_{\K^*}$ the set $g\cdot S\subseteq F^2$ is also a 2-orbit of $G_{\K^*}$. 
Indeed, normality implies that two pairs in the same 2-orbit of $G_{\K^*}$
are mapped by $g$ to two pairs which are also in the same 2-orbit of $G_{\K^*}$, so there exists 
 a 2-orbit $T$ with $g\cdot S\subseteq T$. 
Reasoning analoguously for $g^{-1}$ and $T$ we obtain $g^{-1}\cdot T\subseteq S$, so $g\cdot S=T$.

Call a  2-orbit $S$ of $G_{\K^*}$ {\em black} if $a<^*b$ for all $(a,b)\in S$, and {\em white} if 
$b<^*a$ for all $(a,b)\in S$; orbits which are neither black nor white contain only pairs $(a,b)$ with $a=b$.
Let $S$ be black or white. For every $g\in G_{\K}$, also $g(S)$ is black or white, and if $g(S)$
has the same colour as $S$, then $g(S)=S$. Indeed, as $g\in G_\K$, $g\upharpoonright\{a,b\}$ preserves 
all relations from $L$, and as $g(S)$ has the same colour as $S$, it also preserves $<^*$. Hence, for every $(a,b)\in S$,
$g\upharpoonright\{a,b\}$ is a partial isomorphism of $F^*$, so it extends to some $h\in G_{\K^*}$ 
by homogeneity. Thus $g\cdot (a,b)=h\cdot (a,b)$, so $g\cdot (a,b)\in S$ and $g(S)=S$ follows.

We claim that $g^2\in G_{\K^*}$ for every $g\in G_\K$. Seeking for contradiction, assume that 
there is an $(a,b)\in F^2$ such that $a<^*b$ is not equivalent to
$g^2(a)<^*g^2(b)$. Then there is a  black or white 2-orbit~$S$ of $G_{\K^*}$ such that $g^2(S)$ has a different colour.
The colour of $g(S)$ equals  that of $S$ or $g^2(S)$, and consequently, $S=g(S)$ or $g(S)=g^2(S)$. 
The first case $S=g(S)$ is impossible, because it implies $S=g^2(S)$. The second case $g(S)=g^2(S)$ is also
impossible, because it implies the first via
$g(S)=g^{-1}(g^2(S))=g^{-1}(g(S))=S$.

It follows that  $G_\K/G_{\K^*}$  is an elementary abelian 2-group, i.e., it is the direct product of copies of the 2-element group.
\end{proof}

\begin{theorem}\label{prop:power2}
% Assume that $L$ is relational. Let $\K^*$ be a reasonable ordered Fra\"iss\'e class in the language $L\cup\{<\}$ with the ordering property and 
%with $L$-reduct $\K$. Assume that $\K^*$ is Ramsey.
Let $\K$ be a relational Fra\"iss\'e class with companion $\K^*$. Assume that $G_\K$ is oligomorphic.
Then the following are equivalent.
\begin{enumerate}
\item $|G_{\K}:G_{\K^*}|$ is finite.
\item $|G_{\K}:G_{\K^*}|$ is a finite power of 2.
\item $M(G_{\K})$ is finite.
\item $|M(G_{\K})|$ is a finite power of 2.
\item $G_{\K^*}$ is normal in $G_{\K}$.
\end{enumerate}
\end{theorem}
 
\begin{proof} 
By Theorem~\ref{thm:minflow} we have that $F^*:=\Flim(\K^*)=(F,<^*)$ for $F:=\Flim(\K)$, and that
$M(G_\K)$ is $\overline{G_\K\cdot <^*}$. Then $G_{\K^*}$ 
is oligomorphic by Lemma~\ref{lem:forgetoligo}, and extremely amenable by
Theorem~\ref{thm:extremeamenability}.
In a Hausdorff space a finite set equals its closure.
%a subset is finite if and only if its closure is finite. 
As the elements of $G_\K\cdot <^*$ are in a one-to-
one correspondence with $G_{\K}/G_{\K^*}$, we obtain $(1)\Leftrightarrow(3)$ and $(2)\Leftrightarrow(4)$. 
Proposition~\ref{prop:boolean} implies $(3)\Leftrightarrow(4)$, thus the first four items are equivalent. $(5)\Rightarrow(1)$ 
follows from Lemma~\ref{lem:normalfinite}, and Lemma~\ref{lem:todorpart} implies $(1)\Rightarrow(5)$.
\end{proof}

\begin{corollary}\label{cor:mystery2}
 Let $\K$ be a relational Fra\"iss\'e class that has a companion. Then the 
Ramsey degree for embeddings of $\K$ is infinite or a finite power of 2.
\end{corollary}

\begin{proof} 
By  Proposition~\ref{prop:boolean} and Theorem~\ref{thm:char}.
\end{proof}

\begin{proof}[Proof of Theorem~\ref{thm:mystery}.] By Proposition~\ref{prop:ramseycompanion} and Corollary~\ref{cor:mystery2}.
\end{proof}

\section{Acknowledgements} The authors are indebted to Todor Tsankov for the elegant proof of Proposition~\ref{prop:todor} which is shorter and more general than their original one, and to Manuel Bodirsky, Lionel Nguyen van Th\'{e} and Lyubomyr Zdomskyy for their many comments on the manuscript.

\bibliographystyle{acm}

\end{document}